\documentclass[11pt,reqno,a4paper]{amsart}
\usepackage[margin=1in]{geometry}
\usepackage[usenames]{color}
\usepackage[abbrev,nobysame,alphabetic]{amsrefs}
\usepackage{amsmath,pdfsync,verbatim,graphicx,epstopdf,enumerate}
\usepackage[colorlinks=true]{hyperref}
\usepackage{cancel}
\usepackage{accents}
\usepackage[framemethod=tikz]{mdframed}
\hypersetup{linkcolor=blue,citecolor=red}
\numberwithin{equation}{section}

\usepackage{cleveref}
\newcommand{\lb}[1]{\left[#1\right]}
\newcommand{\I}{\mathrm{i}}

\newcommand{\jd}{j_{\delta}}
\newcommand{\id}{i_{\delta}}

\newcommand{\lr}[1]{\left(#1 \right)}

\newcommand{\rb}{\right)}
\newcommand{\PD}{\partial}
\renewcommand{\d}{\delta}
\newcommand{\Beq}{\begin{equation}}
\newcommand{\Eeq}{\end{equation}}
\newcommand{\beq}{\begin{equation*}}
\newcommand{\eeq}{\end{equation*}}
\newcommand{\bal}{\begin{align}}
\newcommand{\eal}{\end{align}}
\renewcommand{\O}{\Omega}

\newcommand{\N}{\mathcal{N}}

\newcommand{\n}{\nabla}

\usepackage{mathtools}

\newcommand{\bp}{\begin{prob}}
	\newcommand{\ep}{\end{prob}}
\newcommand{\bpr}{\begin{proof}}
	\newcommand{\epr}{\end{proof}}

\newcommand{\tred}[1]{{\color{red}{#1}}}

\newcommand{\bel}[1]{\begin{equation}\label{#1}}
\newcommand{\ee}{\end{equation}}

\newcommand{\NT}{\negthinspace}

\newtheorem{theorem}{Theorem}[section]
\newtheorem{question}{Question}[section]
\newtheorem{corollary}[theorem]{Corollary}

\newtheorem{lemma}[theorem]{Lemma}
\newtheorem{proposition}[theorem]{Proposition}

\theoremstyle{definition}
\newtheorem{definition}[theorem]{Definition}

\newtheorem{remark}[theorem]{Remark}

\newcommand{\Rn}{\mathbb{R}^n}
\newcommand{\R}{\mathbb{R}}
\newcommand{\D}{\mathrm{d}}
\newcommand{\Lc}{\mathcal{L}}
\newcommand{\Rb}{\mathbb{R}}
\newcommand{\Dc}{\mathcal{D}}

\newcommand{\A}{\alpha}
\newcommand{\vp}{\varphi}
\newcommand{\Oc}{\mathcal{O}}

\def\tbl{\textcolor{blue}}

\usepackage{amsmath,pdfsync,verbatim,graphicx,epstopdf,enumerate}

\renewcommand{\d}{\delta}
\newcommand{\wt}{\widetilde}

\newcommand{\Sb}{\mathbb{S}}

\newcommand{\lv}{\lVert}
\newcommand{\rv}{\rVert}

\renewcommand{\O}{\Omega}

\newcommand{\ms}[1]{\begin{quotation}\textbf{\color{teal}Mikko's comment:\
}{\color{teal}\textit{#1}}\end{quotation}}
\newcommand{\ssc}[1]{\begin{quotation}\textbf{\color{blue}Suman's comment:\
}{\color{blue}\textit{#1}}\end{quotation}}

\title[]{The linearized Calder\'on problem for polyharmonic operators}
\author[Sahoo and Salo]{Suman Kumar Sahoo and Mikko Salo}
\address{ Department of Mathematics and Statistics, University of Jyv\"askyl\"a, Finland.
\newline
E-mail:{\tt \ suman.k.sahoo@jyu.fi}}
\address{Department of Mathematics and Statistics, University of Jyv\"askyl\"a, Finland.
\newline
E-mail:{\tt\  mikko.j.salo@jyu.fi}}

\begin{document}
	
\begin{abstract}
	In this article we consider a linearized Calder\'on problem for polyharmonic operators of order $2m\ (m\ge 2)$ in the spirit of Calder\'on's original work \cite{Calderon1980}. We give a uniqueness result for determining coefficients of order $\leq 2m-1$ up to gauge, based on inverting momentum ray transforms.
\end{abstract}

	
	\maketitle
	
	\section{introduction}
	
	Let $n\ge 3$ and let $\O\subset \Rb^n$  be a bounded domain with smooth boundary $\PD\O$.  For $m\ge 2$, we consider the following polyharmonic operator $\Lc$ with lower order anisotropic perturbations  up to order $2m-1$: 
\begin{equation}\label{operator}
		\Lc(x,D)=
		(-\Delta)^m + Q(x,D),
	\end{equation}
	where 
	\begin{equation} \label{operator2}
	    Q(x,D)= \sum\limits_{l=0}^{2m-1} a^{l}_{i_1\cdots i_{l}}(x) \, D^{i_1\cdots i_l}
	\end{equation}
	is a differential operator of order $2m-1$ with $1\le i_1,\cdots,i_l\le n$ and $a^{l}$ is a smooth symmetric  tensor field of order $l$ in $\overline{\Omega}$. Einstein summation convention is assumed for repeated indices throughout the article. 

	The boundary measurements corresponding to the equation $
	\Lc(x,D)u = 0$  in $\Omega $
	may be encoded in terms of the \emph{Cauchy data set} (see e.g.\  \cite{real_principal_inv_prob} for the general case) 
	\begin{align*}
    C_{\Lc} = \{(u|_{\PD\Omega},\PD_{\nu} u|_{\PD\Omega}, \cdots ,\PD^{2m-1}_{\nu} u|_{\PD\Omega}) \,:\, u\in H^{2m}(\Omega),\  \Lc u=0\}.
\end{align*}
    The inverse problem of interest is to determine some information on the coefficients of the operator $\Lc$, up to suitable gauge transformations, from the knowledge of the Cauchy data set $\Lc$.
    
    For various special choices of boundary conditions, one could define a Dirichlet-to-Neumann type operator for solutions of $\Lc u = 0$ in $\Omega$. This requires that the boundary conditions lead to an elliptic boundary value problem (Lopatinskii-Shapiro condition) and that $0$ is not an eigenvalue for this problem. For example, one could consider solutions with the clamped boundary conditions 
    \[
    u|_{\PD \Omega} = f_0, \ \PD_{\nu} u|_{\PD \Omega} = f_1, \ \ldots \ , \ \PD_{\nu}^{m-1} u|_{\PD \Omega} = f_{m-1}
    \]
	and consider the boundary map 
	\[
	\Lambda_{\Lc}^C: (f_0, \ldots, f_{m-1}) \mapsto (\PD_{\nu}^m u|_{\PD \Omega}, \ldots, \PD_{\nu}^{2m-1} u|_{\PD \Omega}).
	\]
	Alternatively, one could consider Navier boundary conditions 
	\[
    u|_{\PD \Omega} = f_0, \ (-\Delta) u|_{\PD \Omega} = f_1, \ \ldots \ , \ (-\Delta)^{m-1} u|_{\PD \Omega} = f_{m-1}
    \]
	and consider the boundary map 
	\[
	\Lambda_{\Lc}^N: (f_0, \ldots, f_{m-1}) \mapsto (\PD_{\nu} u|_{\PD \Omega}, \ \PD_{\nu} (-\Delta) u|_{\PD \Omega}, \ \ldots, \PD_{\nu} (-\Delta)^{m-1} u|_{\PD \Omega}).
	\]
	If $0$ is not an eigenvalue of the corresponding elliptic boundary value problem, then knowing the Cauchy data set is equivalent to knowing the boundary map. One can view the Cauchy data set as a generalization of such boundary maps that is independent of the choice of (elliptic) boundary conditions.

	The most classical case is $m=1$, so that $\Lc$ is a second order elliptic operator. The prototypical inverse problem for such operators is the inverse conductivity problem posed by Calder\'on \cite{Calderon1980}. The original problem was stated for the equation $\mathrm{div}(\gamma \nabla u) = 0$, but if the conductivity function $\gamma$ is $C^2$ and positive one can reduce matters to the Schr\"odinger equation $(\Delta+q) u = 0$ for some potential $q$. One can further consider equations of the form $\Delta u + A(x) \cdot \nabla u + qu = 0$ for some vector field $A$ and potential $q$. This includes the magnetic Schr\"odinger equation and equations with convection terms, which are clearly of the form \eqref{operator}. Various results are known for related inverse problems for determining the vector field $A(x)$, sometimes up to a gauge of the form $A \mapsto A + \nabla \phi$, and the potential $q$ from boundary measurements. We refer the reader to the survey \cite{Uhl_eip_survey,Uhlmann_survey} for more information and references on the second order case.

In this paper we consider a Calder\'on type problem for polyharmonic operators of order $2m$ with a lower order perturbation up to order $2m-1$. 
There are many previous results on inverse problems for polyharmonic operators, including \cite{KRU2,KRU1,Ghosh15,Ghosh-Krishnan,BG_19,BG20,Brown_Gauthier}. In all of these results one only determines lower order coefficients up to order $\leq m$ from boundary measurements. The reason for this restriction is that the method of complex geometrical optics solutions used in these inverse problems requires certain $L^2$ Carleman estimates, which can become highly delicate for higher order operators. In the previous works the required Carleman estimate for the polyharmonic operator has been obtained by using a well known Carleman estimate for second order equations several times in a row. Since there is a loss of half derivative in the original Carleman estimate (it involves a limiting Carleman weight), iterating it many times leads to a loss of several derivatives and thus restricts the method to lower order terms of degree $\leq m$.

In this article we will consider a general higher order elliptic operator given by \eqref{operator} and recover several lower order  coefficients up to order $2m-1$. However, because of the restriction mentioned above, we are not able to consider the full nonlinear inverse problem. Rather, we will only consider the linearization of this inverse problem. In the linearized problem it is sufficient to use solutions of the ``free'' equation $(-\Delta)^m v = 0$ in the recovery of the coefficients, and this avoids the need to use Carleman estimates. However, even in the linearized problem the recovery of higher order terms becomes intricate. We will follow the ideas of  \cite{polyharmonic_mrt_application} where it was observed that momentum ray transforms (MRT) appear naturally in solving inverse problem for polyharmonic operators. The inversion of various MRT is crucial for recovering the coefficients. In particular, to handle the coefficient of order $2m-1$, we study the kernel of a partial MRT. To do so, we demonstrate a new trace free Helmholtz type decomposition result for symmetric tensor fields. 

The rest of the article is organised as follows. In Section \ref{sec:linearization} we derive the linearized inverse problem and state our main results. Section \ref{sec:uniqueness_up_to_2m-2} is devoted to the unique recovery of tensor fields up to order $2m-2$. Then  Section \ref{sec:uniqueness_up_to_2m-1} deals with the recovery of coefficients up to order $2m-1$ under certain assumptions. Section \ref{Gauge transformation} deals with gauge transformations and recovery of coefficients up to natural obstructions; see Theorem \ref{main_theorem_3}. Finally, in Section \ref{sec:mrt} we describe the kernel of MRT which is the key tool for proving  our main results. The required Helmholtz type decomposition result for tensor fields is then proved in  Section \ref{sec:decomposition}.  In Appendix \ref{sec:preliminary} we review some known results and give the construction of special solutions (known as CGO solutions) of $(-\Delta)^m u=0$. The Navier to Neumann map is then linearized in Appendix \ref{linearization} by computing the Fr\'echet derivative. 

\section{Linearization and main results}\label{sec:linearization}

We now derive the formal linearization of the polyharmonic inverse problem. The forward operator is formally given by the map $\Lc \mapsto C_{\Lc}$. Since $C_{\Lc}$ is a set and not necessarily an element in a Banach space, we cannot directly compute the Fr\'echet derivative of the forward operator. We instead assume that
 \begin{align}
      C_{\Lc_{\epsilon}} = C_{\Lc_0} \text{ for all } \epsilon\in (-a,a),\quad a>0,
 \end{align}
 where $ \Lc_0=(-\Delta)^m$, $ \Lc_{\epsilon}= (-\Delta)^m+ Q(x,\epsilon,D)$ and $ Q(x,\epsilon,D) = \sum\limits_{l=0}^{2m-1} a^{l}_{i_1\cdots i_{l}}(x,\epsilon)\, D^{i_1\cdots i_l}$. We assume that $a^{l}$ depend smoothly on $x$ and $\epsilon$ and that $ a^{l}(x,0)=0 $ for all $ 0\le l\le 2m-1$, i.e.\ we are computing the linearization at the case of zero lower order coefficients.
 Using that  $C_{\Lc_{\epsilon}} = C_{\Lc_0}$, from \cite[Lemma 2.8]{real_principal_inv_prob} we have the integral identity 
 \begin{align}\label{useful_identity}
     \langle (\Lc_{\epsilon}-\Lc_0)u,v\rangle_{L^2(\Omega)}=0,
 \end{align}
 for any $u, v \in H^{2m}(\Omega)$ satisfying  
 $ \Lc_{\epsilon} u=0$ and $\Lc_0^*\,v=0$  in $ \Omega$. This can be rewritten as 
 \begin{align*}
  \left  \langle \sum\limits_{l=0}^{2m-1} a^{l}_{i_1\cdots i_{l}}(x,\epsilon)\, D^{i_1\cdots i_l} u(x,\epsilon),v(x) \right\rangle  =0 \text{ for all } \epsilon \in (-a,a).
 \end{align*}
Differentiating this with respect to $ \epsilon$ we obtain
\begin{align*}
   \left  \langle \sum\limits_{l=0}^{2m-1} \PD_{\epsilon}( a^{l}_{i_1\cdots i_{l}}(x,\epsilon))\, D^{i_1\cdots i_l} u(x,\epsilon),v(x) \right\rangle + \left  \langle \sum\limits_{l=0}^{2m-1} a^{l}_{i_1\cdots i_{l}}(x,\epsilon)\, D^{i_1\cdots i_l} \PD_{\epsilon} u(x,\epsilon),v(x) \right\rangle  =0.  
\end{align*}
Writing $w =  u|_{\epsilon = 0}$ and setting $\epsilon =0$, we have 
\begin{align}\label{integral_identity}
   \left  \langle \sum\limits_{l=0}^{2m-1} \PD_{\epsilon}a^{l}_{i_1\cdots i_{l}}(x,0)\, D^{i_1\cdots i_l} w(x),v(x) \right\rangle =0.
\end{align}
Setting $\epsilon=0$ in the equation $\Lc u = 0$ and using $  a^{l}(x,0)=0 $ for all $ 0\le l\le 2m-1$, we see that $w$ solves $(-\Delta)^m w = 0$. It follows that the identity \eqref{integral_identity} holds for any $w$ and $v$ solving $ (-\Delta)^m w = (-\Delta)^m v =0$ in $\Omega$.

We can now formulate (with slightly different notation) the main uniqueness question for the linearized inverse problem considered in this article.

\begin{question} \label{q_linearized}
If $a^l_{i_1 \cdots i_l}$ for $0 \leq l \leq 2m-1$ are smooth tensor fields in $\overline{\Omega}$ and if 
\[
\int_{\Omega} \sum\limits_{l=0}^{2m-1} a^{l}_{i_1\cdots i_{l}} (D^{i_1 \cdots i_l} u) v \,dx = 0
\]
for all $u, v \in H^{2m}(\Omega)$ satisfying $(-\Delta)^m u = (-\Delta)^m v = 0$, is it true that the coefficients $a^l_{i_1 \cdots i_l}$ vanish possibly up to suitable gauge transformations?
\end{question}


We mention that if the Cauchy data set is the graph of suitable Dirichlet-to-Neumann type map, one can make the above formal derivation rigorous and the linearized problem is still given by Question \ref{q_linearized}. See Appendix \ref{linearization} where this is done for the Navier boundary conditions.

We now show that
it is not in general possible to recover all the coefficients in Question \ref{q_linearized} due to the presence of a gauge. One possible gauge is obtained by replacing $\Lc$ by $e^{-\phi} \Lc e^{\phi}$ for a suitable function $\phi$. Note that if
\begin{equation} \label{phi_conditions}
\phi \in C^{2m}(\overline{\Omega}) \text{ and $\PD_{\nu}^j \phi|_{\partial \Omega}=0$ for $0 \leq j \leq 2m-1$},
\end{equation}
then $u$ and $e^{\phi} u$ will have the same Cauchy data up to order $2m-1$. Thus for such functions $\phi$ one has 
\[
C_{\Lc} = C_{e^{-\phi} \Lc e^{\phi}}.
\]
It is easy to compute the highest order terms for the conjugated operator:
\begin{align*}
e^{-\phi} \Lc e^{\phi} u  &=e^{-\phi} ((-\Delta)^m + a^{2m-1}_{i_1 \cdots i_{2m-1}} D^{i_1 \cdots i_{2m-1}} + \ldots )e^{\phi} u \\
  &=  (e^{-\phi} (-\Delta) e^{\phi})^m\, u + e^{-\phi} (a^{2m-1}_{i_1 \cdots i_{2m-1}} D^{i_1 \cdots i_{2m-1}} + \ldots )(e^{\phi} u)\\
  &=  (-1)^m(\Delta\phi + |\nabla \phi|^2+ 2 \nabla\phi\cdot \nabla +\Delta)^m\, u + e^{-\phi} (a^{2m-1}_{i_1 \cdots i_{2m-1}} D^{i_1 \cdots i_{2m-1}} + \ldots )(e^{\phi} u)\\
  &= (-\Delta)^m u+ (a^{2m-1}_{i_1 \cdots i_{2m-1}} D^{i_1 \cdots i_{2m-1}} + (-1)^m 2m \nabla\phi \cdot \nabla (\Delta)^{m-1}) u + \ldots.
\end{align*}
Above $\ldots$ denotes lower order terms. Thus there is always a gauge invariance, where one can add terms of the form $(-1)^m 2m \nabla\phi \cdot \nabla (\Delta)^{m-1}$ to the term of order $2m-1$. There will be corresponding changes in the lower order terms as well.

However, if we only consider operators $\Lc$ for which $a^{2m-1} = 0$, then this type of gauge does not arise: if the term $(-1)^m 2m \nabla\phi \cdot \nabla (\Delta)^{m-1}$ vanishes identically, then necessarily $\nabla \phi = 0$ and hence $\phi = 0$ using the condition $\phi|_{\PD \Omega} = 0$.

Our first main result states that for operators with $a^{2m-1} = 0$, one can recover all the lower order terms completely in the linearized inverse problem.

	\begin{theorem}\label{th:main_theorem_1}
	    Suppose that \begin{align}\label{main_identity_1}
	        \int\limits_{\Omega} \sum_{l=0}^{2m-2}a^{l}_{i_1\cdots i_l} D^{i_1\cdots i_l} u\, v=0\quad \mbox{whenever} \quad  \Delta^m u= \Delta^m v=0.
	    \end{align}
	    Then $ a^{l}=0$ in $\Omega$ for $0\le l\le 2m-2$.
	\end{theorem}

	
	  As an immediate corollary of Theorem \ref{th:main_theorem_1} one  obtains the following density result on certain spaces of symmetric tensor fields. For density results involving  harmonic functions and tensor fields see e.g.\ \cite{Catalin_Ali}.  One may also expect other related density results as a consequence of the method of proof of Theorem \ref{th:main_theorem_1}. See Remark \ref{rmk_density_additional}.
	  
	  \begin{corollary}\label{coro_1}
	  For each even integer $M\ge 0$ consider the set \[\mathcal{A}_M\coloneqq \mbox{span} \left\{ v\,\oplus_{0\le |\alpha|\le M}D^{\alpha} u : \Delta^{\lb{\frac{M+2}{2}}} u= \Delta^{\lb{\frac{M+2}{2}}} v=0 \quad \mbox{in} \quad \Omega  \right\}.\] Then $\mathcal{A}_M$ is dense in  $ \mathbf{S}^M\lr{C^{\infty}(\overline{\Omega})} = \oplus_{p=0}^M S^p\lr{C^{\infty}(\overline{\Omega})}$, where $S^p\lr{C^{\infty}(\overline{\Omega})} $ stands for the space of smooth symmetric  $p$ tensor fields in $ \overline{\Omega}$ and $ \oplus$ is the direct sum.
	  \end{corollary}


	 

Note that for $m=1$, Theorem \ref{th:main_theorem_1} reduces to the statement that if $f\in C^{\infty}(\overline{\Omega})$ satisfies 
\begin{align*} 
    \int\limits_{\Omega} f\, u\, v=0 \quad \mbox{whenever} \quad \Delta u=\Delta v=0 \quad \mbox{in} \quad \Omega,
    \end{align*}
then $f$ vanishes identically in $\Omega$. This is just the linearized Calder\'on problem for the Schr\"odinger equation (linearized at the zero potential), which has been studied in various settings including partial data   \cite{linearized_partial_data, SjostrandUhlmann2016} and Riemannian manifolds \cite{linearzed_complex_manifold,linearized_KLS}. Results for such linearized problems have recently become important in the context of inverse problems for nonlinear PDEs, see e.g.\ \cite{KrupchykUhlmann2020, LLLS2021}.

Now we state our second main result. It also includes terms of order $2m-1$ and gives a complete answer (modulo an assumption for $a^{2m-1}$ on $\partial \Omega$) for the linearized Calder\'on problem for polyharmonic operators when $m=2,3$, showing that one can determine the coefficients uniquely up to the gauge transform $\Lc \to e^{-\phi} \Lc e^{\phi}$. In particular, there are no other gauge invariances in this problem. We also obtain a partial answer when $m \geq 4$.

\begin{theorem}\label{main_theorem_3}
Let $\Lc$ be as in \eqref{operator}--\eqref{operator2} and suppose  $\PD^{r}_{\nu} a^{2m-1}=0$ on $ \PD\Omega$  for all $0\le  r\le 2m-1$. Assume that  \begin{align*}
	        \int\limits_{\Omega} \sum_{l=0}^{2m-1}a^{l}_{i_1\cdots i_l} D^{i_1\cdots i_l}u\, v=0\quad \mbox{whenever} \quad  \Delta^m u= \Delta^m v=0.
	    \end{align*}%
	   	    \begin{itemize}
	   	        \item [1.] For $m=2$ or $m=3$, one has 
	   	        \[
	   	        \Lc = e^{-\phi} (-\Delta)^m e^{\phi}
	   	        \]
	   	        for some $\phi$ satisfying \eqref{phi_conditions}. In particular,  \begin{align*}
	    a^{2m-1}&=i^{m-1}_{\delta} (\nabla\phi)
	    \end{align*}
	    where $ i_{\d} $ is the symmetrization with  Kronecker delta tensor; see \eqref{def_of_i} for the definition.
	    
	    \item[2.] For $ m\ge 4$, if we additionally assume that $ a^{2m-1}= i^{m-1}_{\d} A^1$ for some vector field $ A^1\in C^{\infty}(\overline{\Omega})$, then $\Lc = e^{-\phi} (-\Delta)^m e^{\phi}$ for some $\phi$ satisfying \eqref{phi_conditions}. In particular, 
	    \begin{align*}
	        A^1= \nabla\phi.
	    \end{align*}%
	   	    \end{itemize}
\end{theorem}
\begin{remark}
    In Theorem \ref{main_theorem_3}, the boundary assumption on the coefficient $a^{2m-1}$ can be removed through a boundary determination result.
\end{remark}
In the next section we present the proof of Theorem \ref{th:main_theorem_1}. Then, in Section \ref{sec:uniqueness_up_to_2m-1}, we prove few results (mainly Proposition \ref{main_theorem_2}), which we then employ in Section \ref{Gauge transformation} to demonstrate Theorem \ref{main_theorem_3}.

\section{Proof of Theorem \ref{th:main_theorem_1}}\label{sec:uniqueness_up_to_2m-2}
In this section we demonstrate the proof of Theorem \ref{th:main_theorem_1}. The proof will be based on the trace free decomposition on $ S^m$, the space of symmetric $m$ tensor fields over $\overline{\Omega}$. 
To this end, we define two operators $i_{\delta}:S^m\rightarrow S^{m+2} $ and $j_{\delta}:S^{m}\rightarrow S^{m-2}$ as follows:
\begin{align}
(i_{\delta}f )_{i_1 \cdots i_{m+2}}&\coloneqq \sigma(f_{i_1\cdots i_m} \otimes \delta_{i_{m+1}i_{m+2}})\label{def_of_i}\\
(j_{\delta} f)_{i_1\cdots i_{m-2}}&\coloneqq \sum\limits_{k=1}^n f_{i_1\cdots i_{m-2}kk},\label{def_of_j}
\end{align}
where $\sigma$ denotes the symmetrization of a tensor field. Here $\delta_{ij}$ is the Kronecker delta tensor which is equal to $1$ for $i=j$ and $0$ otherwise. We also define $j_{\delta} f = 0$ for $f \in S^0$ or $f \in S^1$. Note that the operators $ i_{\delta}$ and $j_{\delta}$ are dual to each other with respect to the $L^2$ inner product. Based on this observation we write the trace free decomposition of tensor fields from \cite{Dairbekov_Sharafutdinov,PSU_book} as follows: for $l \geq 2$ one has 
\begin{align}\label{trace_free_decomposition}
    a^{l} =\sum\limits_{k_l=0}^{[\frac{l}{2}]} i^{k_l}_{\delta} b^{l,k_l}
\end{align}
where $b^{l,k_l}$ is a symmetric tensor field of order $l-2k_l$ with $ j_{\delta}b^{l,k_l}=0$. Any $f\in S^m$ satisfies $j_{\delta} f=0$ is known as trace free tensor field. For $l = 0,1$ any tensor field $a^l$ is trace free.

Note that  \eqref{trace_free_decomposition} is an orthogonal decomposition on the space of square integrable symmetric tensor fields over $\overline{\Omega}$, denoted by  $L^2(S^m)$. This implies that   $  a^l=0$ if and only if  $ b^{l,k_l}=0 $ for all $ k_l$ with $ 0\le k_l\le [\frac{l}{2}]$; see  \cite[Equation 6.4.2]{Sharafutdinov_book}. Thus in the remainder of this section we show that under the conditions in Theorem \ref{th:main_theorem_1}, if $a^l$ has the decomposition \eqref{trace_free_decomposition} for $ 0\leq l \leq 2m-2$, then $ b^{l,k_l}=0 $ for all $ k_l$ with $ 0\le k_l\le [\frac{l}{2}]$. This completes the proof of Theorem \ref{th:main_theorem_1}.

\begin{proof}[Proof of Theorem \ref{th:main_theorem_1}]

We insert \eqref{trace_free_decomposition} in the integral identity \eqref{main_identity_1} and obtain
\begin{align}
    0 &= \sum_{l=0}^{2m-2} \sum\limits_{k_l=0}^{[\frac{l}{2}]} (i_{\delta}^{k_l} b^{l,k_l})_{i_1 \cdots i_l} D^{i_1 \cdots i_l} u \, v = \sum_{l=0}^{2m-2} \sum\limits_{k_l=0}^{[\frac{l}{2}]} b^{l,k_l}_{i_1 \cdots i_{l-2k_l}} D^{i_1 \cdots i_{l-2k_l}} (-\Delta)^{k_l} u \, v \label{eq_4.4*}
\end{align}
whenever $\Delta^m u= \Delta^m v=0$. We now recall suitable complex geometrical optics (CGO) solutions to the polyharmonic equation given in \eqref{cgos2},
\begin{align*}
    u(x;h) =& e^{\frac{1}{h}(e_1 + \I \eta_2) \cdot x}\left(a_0(x) + ha_1(x) + \dots + h^{m-1}a_{m-1}(x) + r(x;h)\right)=  e^{\frac{1}{h}(e_1 + \I \eta_2) \cdot x} \tilde{A},\\
v(x;h) =& e^{-\frac{1}{h}(e_1 + \I \eta_2) \cdot x}\left(b_0(x) + hb_1(x) + \dots + h^{m-1}b_{m-1}(x) + \wt{r}(x;h)\right)=  e^{-\frac{1}{h}(e_1 + \I \eta_2) \cdot x}\tilde{B},
\end{align*}
where the error terms $ r(x;h) $ and $\wt{r}(x;h)$ satisfy the estimate  $\lv r(x,h)\rv_{H_{scl}^{2m}}\le c h^m$ and $\lv \wt{r}(x,h)\rv_{H_{scl}^{2m}}\le c h^m$. Moreover, we also have that $ a_j$ and $b_j$ are smooth functions in $ \overline{\Omega}$ for all $j$ with $0\le j\le m-1$.
Substituting the expressions for $u$ and $v$ given above in the identity \eqref{eq_4.4*} and using \eqref{CGO_1} and \eqref{term_H} we obtain
\begin{align}\label{integral_identity_1}
    0 &= \sum_{l=0}^{2m-2} \sum\limits_{k_l=0}^{[\frac{l}{2}]} b^{l,k_l}_{i_1 \cdots i_{l-2k_l}} D^{i_1 \cdots i_{l-2k_l}} \left[ e^{\frac{1}{h} (e_1+i\eta_2) \cdot x} (-\frac{1}{h}T - \Delta)^{k_l} \tilde{A} \right] e^{-\frac{1}{h}(e_1+i\eta_2) \cdot x} \tilde{B} \notag \\
    &= \sum_{l=0}^{2m-2} \sum\limits_{k_l=0}^{[\frac{l}{2}]} b^{l,k_l}_{i_1 \cdots i_{l-2k_l}} (D_{i_1} + \frac{1}{h} (e_1 + i\eta_2)_{i_1}) \cdots (D_{i_{l-2k_l}} + \frac{1}{h} (e_1 + i\eta_2)_{i_{l-2k_l}}) \left[ (-\frac{1}{h}T - \Delta)^{k_l} \tilde{A} \right] \tilde{B} \notag \\
    &= \sum_{l=0}^{2m-2} \sum\limits_{k_l=0}^{[\frac{l}{2}]} \sum_{j=0}^{l-2k_l}\, \int\limits_{\Omega} \binom{l-2k_l}{j}\,\frac{b^{l,k_l}_{i_1\cdots i_{l-2k_l}}}{h^{l-2k_l-j}}  (e_1+\I\eta_2)_{i_1}\cdots (e_1+\I\eta_2)_{i_{l-2k_l-j}} \nonumber\\ &\qquad \times D^{i_{l-2k_{l}-j+1}}\cdots D^{i_{l-2k_l}}( (-\frac{1}{h}T-\Delta)^{k_l} \tilde{A})\, \tilde{B}.
\end{align}

Here $\eta_2$ can be any unit vector with $e_1 \cdot \eta_2 = 0$.
From now on we write $\eta = \eta _2$.
We recover the coefficients one by one by multiplying \eqref{integral_identity_1} with suitable powers of $h$, by letting $h \to 0$ and by inverting momentum ray transforms (MRT, see Lemma \ref{kernel_mrt_sphere_bundle}). This will be done in several steps.\smallskip

\noindent\textbf{Step 1.} 
We first show that 
\begin{align*}
    b^{2m-2,0}=0.
\end{align*}
\smallskip

We start by multiplying \eqref{integral_identity_1} by $h^{2m-2} $.
   Next  utilizing the estimates given in Lemma \ref{lm:cgo_forms}, $ \lv r(x,h)\rv_{H_{scl}^{2m}}\le c h^m$, $\lv \tilde{r}(x,h)\rv_{H_{scl}^{2m}}\le c h^m$ and letting $h\rightarrow 0$,  we see that  only  $  \int\limits_{\Rn} b^{2m-2,0}_{i_1\cdots i_{2m-2}} (e_1+\I\eta)_{i_1}\cdots (e_1+\I\eta)_{i_{2m-2}} a_0\, b_0$ term survives and other terms will be $0$. This implies 
     
\begin{align}\label{identity_for_a_2m-2}
    \int\limits_{\Rn} b^{2m-2,0}_{i_1\cdots i_{2m-2}} (e_1+\I\eta)_{i_1}\cdots (e_1+\I\eta)_{i_{2m-2}} a_0\, b_0=0.
\end{align}
Adapting ideas from \cite{Kenig_Salo_Survey} we  write, 
\begin{align}\label{sum_of_lower_order_tensor}
   &b^{2m-2,0}_{i_1\cdots i_{2m-2}} (e_1+\I\eta)_{i_1}\cdots (e_1+\I\eta)_{i_{2m-2}}\nonumber\\
   &=\sum\limits_{p=0}^{2m-2}(\I)^p\, \binom{2m-2}{p}\, b_{i_1\cdots i_p1\cdots 1}^{2m-2,0} \eta_{i_1} \cdots \eta_{i_p}\nonumber\\
   &= \sum\limits_{p=0}^{2m-2} b_{i_1\cdots i_p}^{2m-2,0,p}\, \eta_{i_1} \cdots \eta_{i_p}, \quad 2\le i_1,\cdots,i_p\le n \quad \mbox{for each} \quad 0\le p \le 2m-2,
\end{align}
where $b_{i_1\cdots i_p}^{2m-2,0,p} := (\I)^p \, \binom{2m-2}{p}\,  b_{i_1\cdots i_p1\cdots 1}^{2m-2,0} $ for each $p$ with $0\le p\le 2m-2$.

 Recall that $ a_0$ and $b_0$ solve the transport equations $ T^ma_0 =T^mb_0=0$. We utilize \eqref{particular_form} and choose the following specific solutions  $a_0$ and $b_0$, where $y_2$ denotes the direction of $\eta$ and $y''=(y_3,\ldots,y_n)$ are orthogonal directions:  \[ a_0= y_2^{m-1}\,g(y'') e^{-\I\lambda(y_1+\I y_2)} \quad \mbox{and} \quad b_0= y_2^{k}\quad \mbox{for all} \quad k\quad \mbox{with}\quad 0\le k\le m-1. \]
Here $\lambda \in \mathbb{R}$ and $g(y'')$ is any smooth function in the $y''$ variable.
 This along with \eqref{identity_for_a_2m-2} and \eqref{sum_of_lower_order_tensor} implies
\begin{align*}
\int\limits_{\mathbb{R}^{n-2}}   \left(\,\, \int\limits_{\mathbb{R}^2} \sum\limits_{p=0}^{2m-2} b_{i_1\cdots i_p}^{2m-2,0,p}(y_1,y_2,y'')\, \eta_{i_1} \cdots \eta_{i_p}\, y^{m-1+k}_2 e^{-\I\lambda(y_1+\I y_2)} dy_1\, dy_2 \right) g(y'') d y''=0.
\end{align*}
Since $g$ can be any smooth function in the $y''$ variable, varying $g$ and fixing $k=m-1$ yields 
    \begin{align}\label{eq_3.7}
    \int\limits_{\mathbb{R}} \sum\limits_{p=0}^{2m-2} \widehat{b}_{i_1\cdots i_p}^{2m-2,0,p}(\lambda,y_2,y'')\, \eta_{i_1} \cdots \eta_{i_p}\, y^{2m-2}_2 e^{\lambda y_2}  dy_2&=0.
\end{align}
Here $\widehat{(\cdot)} $ is the partial Fourier transform in the $x_1$ variable. This notation will be used throughout the rest of the article to denote the partial Fourier transform. 
Next, we set $\lambda=0$ in above to obtain
\begin{align*}
    \int\limits_{\mathbb{R}} \sum\limits_{p=0}^{2m-2} \widehat{b}_{i_1\cdots i_p}^{2m-2,0,p}(0,y_2,y'')\, \eta_{i_1} \cdots \eta_{i_p}\,y^{2m-2}_2 \, dy_2&=0. 
\end{align*}
Recall that $\eta$ can be any unit vector orthogonal to $e_1$, and $y_2$ is the coordinate in direction of $\eta$. Thus the above identity can be interpreted as the vanishing of a certain MRT. By Lemma \ref{kernel_mrt_sphere_bundle} we obtain 
$\widehat{b}^{2m-2,0}(0,x')= 0$. We next argue by induction and assume that 
\begin{align}\label{eq_4.11}
  \frac{\D^r}{\D\lambda^r}\widehat{b}^{2m-2,0}(0,x')=0\quad \mbox{for} \quad 0\le r\le M. 
\end{align}
Now differentiating \eqref{eq_3.7} with respect to $\lambda$ for $M+1$ times, then  setting $\lambda=0$ and  using \eqref{eq_4.11} we obtain
\begin{align*}
    \int\limits_{\mathbb{R}}  \sum\limits_{p=0}^{2m-2} \left( \frac{\D^{M+1}}{\D \lambda^{M+1}}  \widehat{b}_{i_1\cdots i_p}^{2m-2,0,p}\right)(0,y_2,y'') \,  \eta_{i_1} \cdots \eta_{i_p}\, y^{2m-2}_2 \, dy_2=0.
\end{align*}
The combination of this with Lemma  \ref{kernel_mrt_sphere_bundle} implies
$ \frac{\D^{M+1}}{\D\lambda^{M+1}} \widehat{b}^{2m-2,0}(0,x')=0.
$
Hence by induction we have \begin{align}\label{vanishing_derivative_lambda}
   \frac{\D^{r}}{\D\lambda^{r}} \widehat{b}^{2m-2,0}(0,x')=0 \quad \mbox{ for all non-negative integers $r$}.
\end{align} 
Since $b^{2m-2,0}$ is compactly supported, the function   $\widehat{b}^{2m-2,0}(\lambda,x')$ is analytic in the $ \lambda$ variable by the Paley-Wiener theorem. This together with \eqref{vanishing_derivative_lambda} gives  $b^{2m-2,0}=0.$ \smallskip

\textbf{Step 2.} 
In this step we prove that for each $r$ with $0\le r\le m-2$, we have  
\begin{align}\label{claim_1}
    b^{2m-2-k,r-k}=0\quad \mbox{for all} \quad 0\le k\le r.
\end{align}
\smallskip

 The proof is based on induction in $r$.
 For $r=0$ \eqref{claim_1} follows from  \textbf{Step 1}.
   Assume that \eqref{claim_1}  is true for all $s$ with $0\le s\le r < m-2$. Our goal is to prove \eqref{claim_1} for $r+1$.
  To this end, we first separate the integrals for $j=0$ and that of for $j\ge 1$ in \eqref{integral_identity_1}. Then we again separate the term corresponding to $j=0$ into two terms. As a result from \eqref{integral_identity_1} we obtain,
\begin{align*}
    0= &\nonumber \sum_{l=0}^{2m-2-r-1} \sum\limits_{k_l=0}^{[\frac{l}{2}]} \, \int\limits_{\Omega} \,\frac{b^{l,k_l}_{i_1\cdots i_{l-2k_l}}}{h^{l-2k_l}}  (e_1+\I\eta_2)_{i_1}\cdots (e_1+\I\eta_2)_{i_{l-2k_l}} ( (-\frac{1}{h}T-\Delta)^{k_l} \tilde{A})\, \tilde{B} \nonumber\\
   &\quad+\sum_{l=2m-2-r}^{2m-2} \sum\limits_{k_l=0}^{[\frac{l}{2}]} \, \int\limits_{\Omega} \,\frac{b^{l,k_l}_{i_1\cdots i_{l-2k_l}}}{h^{l-2k_l}}  (e_1+\I\eta_2)_{i_1}\cdots (e_1+\I\eta_2)_{i_{l-2k_l}} ( (-\frac{1}{h}T-\Delta)^{k_l} \tilde{A})\, \tilde{B}\nonumber\\ &\quad + \sum_{l=0}^{2m-2} \sum\limits_{k_l=0}^{[\frac{l}{2}]} \sum_{j=1}^{l-2k_l}\, \int\limits_{\Omega} \binom{l-2k_l}{j}\,\frac{b^{l,k_l}_{i_1\cdots i_{l-2k_l}}}{h^{l-2k_l-j}}  (e_1+\I\eta_2)_{i_1}\cdots (e_1+\I\eta_2)_{i_{l-2k_l-j}} \nonumber\\ &\qquad \times D^{i_{l-2k_{l}-j+1}}\cdots D^{i_{l-2k_l}}( (-\frac{1}{h}T-\Delta)^{k_l} \tilde{A})\, \tilde{B}. 
\end{align*}
Simplifying and utilizing induction assumption in the second integral above we conclude,
\begin{align}\label{eq_3.12}
    0= &\nonumber \sum_{l=0}^{2m-2-r-1} \sum\limits_{k_l=0}^{[\frac{l}{2}]} \, \int\limits_{\Omega} \,\frac{b^{l,k_l}_{i_1\cdots i_{l-2k_l}}}{h^{l-2k_l}}  (e_1+\I\eta_2)_{i_1}\cdots (e_1+\I\eta_2)_{i_{l-2k_l}} ( (-\frac{1}{h}T-\Delta)^{k_l} \tilde{A})\, \tilde{B} \\ \nonumber
   &\quad+
   \sum\limits_{k_{2m-2-r}=1}^{[\frac{2m-2-r}{2}]} \, \int\limits_{\Omega} \,\frac{b^{2m-2-r,k_{2m-2-r}}_{i_1\cdots i_{l-2k_l}}}{h^{2m-2-r-2k_{2m-2-r}}}  (e_1+\I\eta_2)_{i_1}\cdots (e_1+\I\eta_2)_{i_{2m2-2-r-2k_{2m-2-r}}} ( (-\frac{1}{h}T-\Delta)^{k_{2m-2-r}} \tilde{A})\, \tilde{B}\\ &\quad +\cdots+\sum\limits_{k_{2m-2}=r+1}^{m-1} \, \int\limits_{\Omega} \,\frac{b^{2m-2,k_{2m-2}}_{i_1\cdots i_{2m-2-2k_{2m-2}}}}{h^{2m-2-2k_{2m-2}}}  (e_1+\I\eta_2)_{i_1}\cdots (e_1+\I\eta_2)_{i_{2m-2-2k_{2m-2}}} ( (-\frac{1}{h}T-\Delta)^{k_{2m-2}} \tilde{A})\, \tilde{B}\nonumber\\&\quad+ \sum_{l=0}^{2m-2} \sum\limits_{k_l=0}^{[\frac{l}{2}]} \sum_{j=1}^{l-2k_l}\, \int\limits_{\Omega} \binom{l-2k_l}{j}\,\frac{b^{l,k_l}_{i_1\cdots i_{l-2k_l}}}{h^{l-2k_l-j}}  (e_1+\I\eta_2)_{i_1}\cdots (e_1+\I\eta_2)_{i_{l-2k_l-j}} \nonumber\\ &\qquad \times D^{i_{l-2k_{l}-j+1}}\cdots D^{i_{l-2k_l}}( (-\frac{1}{h}T-\Delta)^{k_l} \tilde{A})\, \tilde{B}. 
\end{align}
Our next goal is to find the limit $ \lim_{h\rightarrow 0} h^{2m-2-r-1}\times$ \eqref{eq_3.12}.
To achieve this, we consider the non-zero integrals in the R.H.S of \eqref{eq_3.12} with  $h^{(2m-2-r-1)}$ in the denominator. Observe that, these terms are derived  from first three integrals in the preceding equation (for certain values of summation variables) and  there will be no contribution from the last term in the above displayed equation as we see below. We show this by examining the last term for each values of $j$, when $l\ge 2m-2-r$ as $j\ge 1$. We   verify this only for $j=1$ and for other values of $j$ similar arguments  work. Therefore we consider
\begin{align*}
   & \sum_{l=0}^{2m-2} \sum\limits_{k_l=0}^{[\frac{l}{2}]} \, \int\limits_{\Omega} \binom{l-2k_l}{1}\,\frac{b^{l,k_l}_{i_1\cdots i_{l-2k_l}}}{h^{l-2k_l-1}}  (e_1+\I\eta_2)_{i_1}\cdots (e_1+\I\eta_2)_{i_{l-2k_l-1}}  D^{i_{l-2k_l}}( (-\frac{1}{h}T-\Delta)^{k_l} \tilde{A})\, \tilde{B}\\
    &\quad= \sum_{l=0}^{2m-2-r-1} \sum\limits_{k_l=0}^{[\frac{l}{2}]} \, \int\limits_{\Omega} \binom{l-2k_l}{1}\,\frac{b^{l,k_l}_{i_1\cdots i_{l-2k_l}}}{h^{l-2k_l-1}}  (e_1+\I\eta_2)_{i_1}\cdots (e_1+\I\eta_2)_{i_{l-2k_l-1}}  D^{i_{l-2k_l}}( (-\frac{1}{h}T-\Delta)^{k_l} \tilde{A})\, \tilde{B}\\ & \qquad+ \sum_{l=2m-2-r}^{2m-2} \sum\limits_{k_l=0}^{[\frac{l}{2}]} \, \int\limits_{\Omega} \binom{l-2k_l}{1}\,\frac{b^{l,k_l}_{i_1\cdots i_{l-2k_l}}}{h^{l-2k_l-1}}  (e_1+\I\eta_2)_{i_1}\cdots (e_1+\I\eta_2)_{i_{l-2k_l-1}} D^{i_{l-2k_l}}( (-\frac{1}{h}T-\Delta)^{k_l} \tilde{A})\, \tilde{B}\\
   &\quad= \sum_{l=0}^{2m-2-r-1} \sum\limits_{k_l=0}^{[\frac{l}{2}]} \, \int\limits_{\Omega} \binom{l-2k_l}{1}\,\frac{b^{l,k_l}_{i_1\cdots i_{l-2k_l}}}{h^{l-2k_l-1}}  (e_1+\I\eta_2)_{i_1}\cdots (e_1+\I\eta_2)_{i_{l-2k_l-1}}  D^{i_{l-2k_l}}( (-\frac{1}{h}T-\Delta)^{k_l} \tilde{A})\, \tilde{B}\\ & \qquad+  \sum\limits_{k_{2m-2-r}=1}^{[\frac{2m-2-r}{2}]} \, \int\limits_{\Omega} \binom{2m-2-r-2k_{2m-2-r}}{1}\,\frac{b^{2m-2-r,k_{2m-2-r}}_{i_1\cdots i_{m-2-r-2k_{2m-2-r}}}}{h^{2m-2-r-2k_{2m-2-r}-1}}\\& \qquad \quad \times (e_1+\I\eta_2)_{i_1}\cdots (e_1+\I\eta_2)_{i_{2m-2-r-2k_{2m-2-r}-1}}  D^{i_{2m-2-r-2k_{2m-2-r}}}( (-\frac{1}{h}T-\Delta)^{k_{2m-2-r}} \tilde{A})\, \tilde{B}+ \cdots\\& \qquad+  \sum\limits_{k_{2m-2}=r+1}^{m-1} \, \int\limits_{\Omega} \binom{2m-2-2k_{2m-2}}{1}\,\frac{b^{2m-2,k_{2m-2}}_{i_1\cdots i_{2m-2-2k_{2m-2}}}}{h^{2m-2-2k_{2m-2}-1}}  (e_1+\I\eta_2)_{i_1}\cdots (e_1+\I\eta_2)_{i_{2m-2-2k_{2m-2}-1}}\\ &\qquad \quad\times  D^{i_{2m-2-2k_{2m-2}}}( (-\frac{1}{h}T-\Delta)^{k_{2m-2}} \tilde{A})\, \tilde{B}.
\end{align*}
We arrived at the last equality using the induction hypotheses. Now notice that the highest power of $h$ in the denominator of  the aforementioned terms is $2m-2-r-2$, therefore  these  terms will disappear, when multiplying them by $h^{2m-2-r-1}$ and then letting $h\rightarrow 0$.   
 Thus multiplying  \eqref{eq_3.12} by $h^{2m-2-r-1}$ and letting $h\rightarrow 0$, we obtain 
\begin{align}\label{eq_3.14}
  \sum\limits_{k=0}^{r+1}  \int\limits_{\Omega}  b^{2m-2-k,r+1-k}\,(e_1+\I\eta)_{i_1}\cdots (e_1+\I\eta)_{i_{2m-2-2(r+1)+k}} \left((-T)^{r+1-k}a_0\right)\,b_0=0,\end{align}
Employing the same strategy as in \eqref{sum_of_lower_order_tensor}, we next express $b^{2m-2-k,r+1-k}\,(e_1+\I\eta)_{i_1}\cdots (e_1+\I\eta)_{i_{2m-2-2(r+1)+k}}$ for each $0\le k \le r+1$, as the sum of lower order tensors in the following way:
\begin{align*}
    B^{2m-2-k,r+1-k}=\sum\limits_{p=0}^{2m-2-2(r+1)+k} b_{i_1\cdots i_p}^{2m-2-l,r+1-k,p} \, \eta_{i_1}\cdots \eta_{i_p}\, \quad \mbox{for}\quad 0\le k \le r+1.
\end{align*}%
Next, we choose
$ a_0=y_2^{m-1}\,g(y'') e^{-\I\lambda(y_1+\I y_2)}$ and $ b_0= y_2^{m-1-p} $ for $ 0\le p\le r+1 $ to be the solution of the transport equation $T^m(\cdot)=0,$ where $ g$ is an arbitrary smooth function in $y''$. 
Plugging this into  \eqref{eq_3.14}, then utilizing arbitrariness of $g$ and performing integral in the $y_1$ variable from  \eqref{eq_3.14} we conclude
\begin{align}\label{eq_3.16}
\sum\limits_{k=0}^{r+1}   \int\limits_{\mathbb{R}} c_{r+1-k}  \widehat{B}^{2m-2-k,r+1-k}(\lambda,y_2,y'')\, y^{2m-2-2(r+1)+k+(r+1-p)}_2 e^{\lambda y_2}  dy_2&=0\quad \mbox{for}\quad 0\le p\le r+1
\end{align}
where $ c_{r+1-k}$ are non-zero constants for $ 0\le k \le r+1$. Note that the constant $(-1)^{r+1-k}$ appearing in the equation \eqref{eq_3.14} is absorbed into $c_{r+1-k}$. 
Next, we specify $\lambda=0$ in above to obtain
\begin{align*}
  \sum\limits_{k=0}^{r+1}c_{r+1-k}  I^{2m-2-2(r+1)+k+(r+1-p)} \widehat{B}^{2m-2-k,r+1-k}=0 \quad \mbox{for}\quad 0\le p\le r+1, 
\end{align*}
where for all non-negative integers $m$,  $I^m$ stands for the MRT defined in Section \ref{sec:mrt}; see equation \eqref{Eq4.1}.
We next combine Lemma \ref{lem:sum_of_mrt} with \Cref{kernel_mrt_sphere_bundle} to deduce that $\widehat{b}^{2m-2-k,r+1-k}(0,x')=0$ for all $0\le k\le r+1.$ At first, applying Lemma \ref{lem:sum_of_mrt} from the preceding equation we segregate MRT of different order tensor fields. This means that for any $ 0\le k\le r+1 $, we have  $ I^{2m-2-2(r+1)+k} \widehat{B}^{2m-2-k,r+1-k}=0 $. This along with Lemma \ref{kernel_mrt_sphere_bundle} entails $\widehat{b}^{2m-2-k,r+1-k}(0,x')=0$ for all $0\le k\le r+1.$

We next argue by induction and  for each $ 0\le k\le r+1$ assume that 
\begin{align}\label{eq_3.18}
 \left( \frac{\D^p}{\D \lambda^p} \widehat{b}^{2m-2-k,r+1-k}\right)(0,x')=0 \quad \mbox{for all} \quad 0\le p \le M.
\end{align}
Next  differentiating \eqref{eq_3.16} $M+1$ times more and then evaluating at $\lambda=0$ and using \eqref{eq_3.18} we obtain
\begin{align*}
   \sum\limits_{k=0}^{r+1}c_{r+1-k}  I^{2m-2-2(r+1)+k+(r+1-p)} \left(\frac{\D^{M+1}}{\D \lambda^{M+1}}\widehat{B}^{2m-2-k,r+1-k}\right)=0.
\end{align*}
This together with Lemma \ref{lem:sum_of_mrt} entails
\begin{align*}
    I^{2m-2-2(r+1)+k} \left(\frac{\D^{M+1}}{\D \lambda^{M+1}}\widehat{B}^{2m-2-k,r+1-k}\right)=0 \quad \mbox{for all} \quad 0\le k\le r+1.
\end{align*}
By  Lemma \ref{kernel_mrt_sphere_bundle} this gives  $
 \left(\frac{\D^{M+1}}  {\D \lambda^{M+1}}\widehat{b}^{2m-2-k,r+1-k}\right) (0,x')=0 $ for all $ 0\le k\le r+1.
$
Hence by induction for all non-negative integers $p$, we obtain 
\begin{align}\label{eq_3.20}
    \left(\frac{\D^{p}}  {\D \lambda^{p}}\widehat{b}^{2m-2-k,r+1-k}\right) (0,x')=0 \quad \mbox{for all} \quad 0\le k\le r+1.
\end{align}
We have that for each $ 0\le k\le r+1$, the tensor $b^{2m-2-k,r+1-k}$ is compactly supported in the $x_1$ variable. This implies  the tensor  $\widehat{b}^{2m-2-k,r+1-k}(\lambda,x')$ is analytic in the $ \lambda$ variable by the Paley-Wiener theorem. This together with \eqref{eq_3.20} implies that, for each $ 0\le k\le r+1$ we have
$ b^{2m-2-k,r+1-k}=0 \quad \mbox{in} \quad \Omega.
$ This finished the induction step and completes the proof of  \eqref{claim_1}.

Observe that, $ m\le 2m-2-k\le 2m-2 $ and $  0\le r-k\le m-2-k$ whenever $ 0\le k\le r$ and $ 0\le r\le m-2$. After a re-indexing we conclude combining  \textbf{Step 1} and \textbf{Step 2}  that
\begin{align*}
     b^{l,k_l}=0 \quad \mbox{for all} \quad k_l \quad \mbox{such that} \quad 0\le k_l\le l-m \quad \mbox{and} \quad m\le l\le 2m-2.
\end{align*}\smallskip
At this stage, inserting above findings into the integral identity \eqref{eq_4.4*} we obtain 
\begin{align*}
    0= \sum_{l=m}^{2m-2} \sum\limits_{k_l=l-m+1}^{[\frac{l}{2}]} b^{l,k_l}_{i_1 \cdots i_{l-2k_l}} D^{i_1 \cdots i_{l-2k_l}} (-\Delta)^{k_l} u \, v+  \sum_{l=0}^{m-1} \sum\limits_{k_l=0}^{[\frac{l}{2}]} b^{l,k_l}_{i_1 \cdots i_{l-2k_l}} D^{i_1 \cdots i_{l-2k_l}} (-\Delta)^{k_l} u \, v. 
\end{align*}
\smallskip

\textbf{Step 3.} Next, we show that 
 \begin{align}\label{recover_up_to_m-1}
     b^{l,k_l}=\begin{cases}
     0 \quad \mbox{for all} \quad 0 \le k_l\le \left[\frac{l}{2}\right] \quad \hspace{6.8mm} \mbox{where}\quad  0\le l\le m-1\\
     0 \quad\mbox{for all}\quad l-m< k_l\le \left[\frac{l}{2}\right]\quad \mbox{where} \quad m\le l \le 2m-2.
     \end{cases}
 \end{align}\\
 
 This step can be carried out using the same analysis used in \cite[Theorem 1.1]{polyharmonic_mrt_application}. However, we give an alternate proof using the ideas  from \textbf{Step 2}.
   To this end, we  multiply \eqref{integral_identity_1} by $h^{m-1}$ and let $ h\rightarrow 0$ (using similar arguments given in \textbf{Step 2}) to obtain
  \begin{align}\label{eq_4.19}
  \sum\limits_{l=m-1}^{2m-2}  \int\limits_{\Omega} b^{l,k_l}\,(e_1+\I\eta)_{i_1}\cdots (e_1+\I\eta)_{i_{l-2k_l}} T^{k_l}\,a_0\,b_0=0,
\end{align}
where $ m-1-k_{m-1}=\cdots=2m-2-k_{2m-2}= A=m-1 $ and $a_0$, $b_0$ satisfy  $T^m (\cdot)=0$. This entails
 \begin{align*}
  \sum\limits_{l=m-1}^{2m-2}  \int\limits_{\Omega} b^{l,l-(m-1)}\,(e_1+\I\eta)_{i_1}\cdots (e_1+\I\eta)_{i_{2m-2-l}} T^{l-(m-1)}\,a_0\,b_0=0\quad \mbox{with} \quad T^ma_0 =T^mb_0=0.
\end{align*}
Extending the tensor fields $ b^{l,l-(m-1)}$  by $0$ outside of $\overline{\Omega}$, then choosing $ a_0=y_2^{m-1}\,g(y'') e^{-\I\lambda(y_1+\I y_2)}$ and $ b_0= y_2^{m-1-p} $ for $ 0\le p\le m-1 $, varying  $g$ to be any smooth function in the $y''$ variable and performing integral in the $y_1$ variable,  we obtain for almost every $y''$ that      
\begin{align*}
    \sum\limits_{l=m-1}^{2m-2} c_{l-(m-1)}\, e^{\lambda y_2}\, I^{2m-2-l+(m-1-p)} \widehat{B}^{l,l-(m-1)}(\lambda,y_2,y'')d y_2=0 \quad \mbox{for all $p$ with}\quad 0\le p\le m-1,
\end{align*}
where $B^{l,l-(m-1)} $ can be expressed in terms of $b^{l,l-(m-1)}$ as it was done in  \eqref{sum_of_lower_order_tensor}. At this stage, utilizing Lemma \ref{lem:sum_of_mrt} we deduce that $  I^{2m-2-l} \widehat{B}^{l,l-(m-1)}=0 $ for every $l$ satisfying $ m-1\le l\le 2m-2. $
This along with Lemma \ref{kernel_mrt_sphere_bundle} yields
$ \widehat{b}^{l,l-(m-1)}(0,x')=0$ for $ m-1 \le l\le 2m-2$. We now use induction to conclude that for all non-negative integers $p$, $ \frac{\D^p }{\D \lambda^p}\widehat{b}^{l,l-(m-1)}(\lambda,x')=0$ , where $ m-1 \le l\le 2m-2 $, which is similar to the line of argument presented in the previous steps; see for instance  [\eqref{eq_4.11}-\eqref{vanishing_derivative_lambda}]. Combining this with Paley-Wiener theorem we obtain $ b^{l,l-(m-1)}(x)=0$ , where $ m-1 \le l\le 2m-2 $.

Next we consider the coefficients multiplying \eqref{integral_identity_1} by  $ h^{m-2}$ and letting $h\rightarrow 0$. In this case, we will obtain the next  identity by replacing $ m$ by $ m-1$ in \eqref{eq_4.19}.  \begin{align*}
  \sum\limits_{l=m-2}^{2m-4}  \int\limits_{\Omega} b^{l,k_l}\,(e_1+\I\eta)_{i_1}\cdots (e_1+\I\eta)_{i_{l-2k_l}} T^{k_l}\,a_0\,b_0=0,
\end{align*}%
where $ m-2-k_{m-2}=\cdots=2m-4-k_{2m-4}= A=m-2 $. Iterating similar steps as above we can infer that $ b^{l,l-(m-2)}=0$ for all $l $ with $ m-2\le l \le 2m-4$.
In doing so, after finitely many steps, we conclude  
\eqref{recover_up_to_m-1}. Now the proof of Theorem \ref{th:main_theorem_1} is complete combining  all the steps (\textbf{Step 1},\textbf{ Step 2} and \textbf{ Step 3}).
\end{proof}
 \begin{lemma}\label{kernel_mrt_sphere_bundle}
    Let $f(x_1,x')$ be a bounded symmetric $m$-tensor field in $\Rn$, compactly supported in $x'$ variable. Suppose for all unit vectors $ \eta\perp  e_1$, writing $x' = (x_2,x'')$ where $x_2$ is the direction of $\eta$ and $x''$ are orthogonal directions, we have that
    \begin{align*}
       \int\limits_{\Rb} x_2^{m}\, f_{i_1\cdots i_m}(0,x_2,x'') \left(\prod_{j=1}^{m}(e_1+\I \eta)_{i_j}\right) \D x_2& = 0, \quad \mbox{for a.e. } x''.
    \end{align*}
 Then, 
 \begin{align*}
    f(0,x')=   \begin{cases}
        & i_{\d} v(0,x') \quad \mbox{for} \quad m\ge 2\\
         & 0 \quad \hspace{12mm} \mbox{for} \quad m =0,1
     \end{cases}
 \end{align*}
 where $v$ is a symmetric $m-2$ tensor field  compactly supported in $x'$ variable. 
In addition, if we assume that $f$ satisfies $j_{\d}f=0$  then $ f(0,x')=0.$
 
\end{lemma}
 We do not present the proof of this lemma as it follows from \cite[Lemma 3.7]{polyharmonic_mrt_application}. We end this section with the following  remarks. 
 
 \begin{remark}
 This remark suggests that there is a another  way of avoiding the gauge appearing in the Theorem \ref{main_theorem_3}. 
 Instead of assuming $ a^{2m-1}=0$, one can assume that  $ a^{l}=i_{\d} b^{l-2}$ for $ m\le l \le 2m-1$ in Theorem \ref{th:main_theorem_1} and conclude in a similar way that 
 %
       $ a^{j}=0$ for all $j$ with $  0\le j \le 2m-1. $
    
    
 \end{remark}
 
	     \begin{remark} \label{rmk_density_additional}
	   Corollary \ref{coro_1} gave a density result for even integers. For odd integers $K \ge 0$, Theorem \ref{th:main_theorem_1} implies the following (non-optimal) result: the set \[\mathcal{A}_K= \mbox{span} \left \{ v\,\oplus_{|\alpha|\le K}D^{\alpha} u : \Delta^{\lb{\frac{K+2}{2}}+1} u= \Delta^{\lb{\frac{K+2}{2}}+1} v=0 \quad \mbox{in} \quad \Omega  \right\}\] is dense in the space $ \mathbf{S}^K\lr{C^{\infty}(\overline{\Omega})} = \oplus_{p=0}^K S^p\lr{C^{\infty}(\overline{\Omega})}$. 
	  
	   One could also ask if the linear span of the  set 
	   $$ \mathcal{A}_{p_1p_2}\coloneqq \left\{ D^{\alpha}v \otimes D^{\beta} u :   \Delta^{p_1+1} u= \Delta^{p_2+1} v=0 \,\, \mbox{in} \,\, \Omega,\,\,\mbox{for all multi-indices $\alpha$, $\beta$ such that}\,\, \lvert \alpha \rvert =p_1, \lvert \beta \rvert =p_2 \right\},$$
	   is dense in the space smooth symmetric tensor fields defined in $\overline{\Omega} $ of order $p_1+p_2$. Here $\otimes$ denotes the symmetric tensor product. 
	   
	   In particular when  $p_1=p_2=1$, this is asking if the linear span of tensor products of gradients of any two biharmonic functions is dense in the space of smooth symmetric two tensor fields. However, a similar density result is not true if one replace biharmonic functions by harmonic functions. The latter density question arises in the  linearized version of the \emph{boundary rigidity problem}; see \cite[Chapter 1]{Sharafutdinov_book} and \cite[Chapter 11]{PSU_book}. This question also appears while linearizing certain anisotropic elliptic PDE from the Cauchy data set; see \cite[Section 6]{Catalin_Ali}.
	   \end{remark}
 
 
 \section{Recovery of coefficients under an additional assumption }\label{sec:uniqueness_up_to_2m-1}
In this section we present the proof of Proposition \ref{main_theorem_2}  which will be the main ingredient to prove  Theorem \ref{main_theorem_3} in the following section.
To proceed further, we define two differential operators  on $C^{\infty}(S^m)$, the space of smooth symmetric tensor fields of order $m$. 
\begin{definition}\cite[Chapter 2]{Sharafutdinov_book}
The symmetric covariant derivative $\D: C^{\infty}(S^m)\rightarrow C^{\infty}(S^{m+1})$ is given by
\begin{align*}
    (\D f)_{i_1\cdots i_{m+1}} =\sigma(i_1\cdots i_{m+1}) \frac{\PD f_{i_1\cdots i_m}}{\PD x_{i_{m+1}}}, \quad 1\le i_1\cdots ,i_{m+1}\le n.
\end{align*}
\end{definition} 
\begin{definition}\cite[Chapter 2]{Sharafutdinov_book}
 The divergence operator $\d: C^{\infty}(S^{m+1})\rightarrow C^{\infty}(S^m)$ is given by
 \begin{align*}
(\d f)_{i_1\cdots i_{m}} =\sum\limits_{k=1}^{n}  \frac{\PD f_{i_1\cdots i_m k }}{\PD x_{k}}, \quad 1\le i_1\cdots ,i_{m}\le n.
 \end{align*}
\end{definition}
\noindent The operators $-\D$ and $\delta$ are dual to each other with respect to $L^2$ inner product. We also define $\D^k: C^{\infty}(S^m)\rightarrow C^{\infty}(S^{m+k}) $ by taking composition of $\D$ with itself $k$ times. Similarly we define $\delta^k: C^{\infty}(S^m)\rightarrow C^{\infty}(S^{m-k})$. In particular, 
\begin{align*}
    \delta^m f&= \sum\limits_{i_1,\cdots, i_m=1}^{n}\frac{\PD^m f_{i_1\cdots i_m}}{\PD x_{i_1}\cdots \PD x_{i_m}}, \quad \mbox{for any}\quad f\in C^{\infty}(S^m), \\
    \D^m\phi&= \frac{\PD^m \phi}{\PD x_{i_1}\cdots \PD x_{i_m}},\quad \qquad \qquad\mbox{for any}\quad \phi \in C^{\infty}(\overline{\Omega}).
\end{align*}
Our next result is as follows.

\begin{proposition}\label{main_theorem_2}
  Let $m \ge 1$. Suppose  
   \begin{align}\label{main_identity_4}
	        \int\limits_{\Omega} \sum_{l=0}^{2m-1}a^{l}_{i_1\cdots i_l} D^{i_1\cdots i_l}u\, v=0\quad \mbox{whenever} \quad  \Delta^m u= \Delta^m v=0. 
	    \end{align}
	\begin{itemize}
	    \item [1.] Then we have $ a^{2m-1}=0$ in $ \Omega$ if we additionally  assume
	    \begin{align}\label{div_trace_free-assumption}
      \delta^{2m-1} a^{2m-1}=0 \quad\mbox{and} \quad j_{\delta} a^{2m-1}=0.
  \end{align} As a result, this implies $ a^{j}=0$ in $ \Omega $ for $ 0\le j \le 2m-1$ by Theorem \ref{th:main_theorem_1}.
  \item[2.] Moreover, we also have  \begin{align}\label{highest_order_tensor}
	        a^{2m-1}= \D^{2m-1}\phi + i_{\d} a^{2m-1,1} \quad \mbox{where $\phi$ fulfils}\quad \PD^{l}_{\nu} \phi|_{\partial \Omega}=0\quad \mbox{for}\quad 0\le l\le 2m-2,
	        \end{align}
	    where $\phi$ is a scalar function and $a^{2m-1,1} $ is a symmetric tensor of order $2m-3$.
	\end{itemize}
\end{proposition}



\begin{proof}
\textbf{Step 1.} Extend $a^{2m-1}$ by zero to $\Rn$. We first prove that  \eqref{main_identity_4} implies
 \begin{align}\label{main_identity}
	       \int\limits_{\Rn} a^{2m-1}_{i_1\cdots i_{2m-1}} (e_1+\I\eta)_{i_1}\cdots (e_1+\I\eta)_{i_{2m-1}} a_0\, b_0=0, \quad \mbox{whenever} \quad T^ma_0= T^mb_0=0.
	    \end{align}\smallskip

This step can be proved by choosing $ u$ and $v$ to be CGO solutions as in \eqref{cgos2} and then multiplying \eqref{main_identity_4} by $h^{2m-1}$ and letting $ h\rightarrow 0$.\\

 \textbf{Step 2.} Here we show that \eqref{main_identity} implies that certain MRTs of 
    \begin{align*}
             \widehat{a}^{2m-1,odd} \coloneqq \sum\limits_{p=0}^{m-1} i'^{m-1-p}_{\d} \widehat{a}^{2m-1,2p+1}(0,y_2,y'')
           \quad \mbox{and}\quad
    \widehat{a}^{2m-1,even}  \coloneqq \sum\limits_{p=0}^{m-1} i'^{m-p}_{\d} \widehat{a}^{2m-1,2p}(0,y_2,y'') 
    \end{align*}
become $0$.
Recall that $\hat{a}$ denotes the Fourier transform of $a$ in $x_1$, and note that $\widehat{a}^{2m-1,odd}$ and $\widehat{a}^{2m-1,even}$ are tensor fields of order $2m-1$ and $2m-2$ respectively. Also $i'_{\d}$ and $j'_{\d}$ are  same as \eqref{def_of_i} and \eqref{def_of_j} respectively. However in this case the indices vary from $2$ to $n$.
\\

Similar to \eqref{sum_of_lower_order_tensor}) we next  write 
 \begin{align*}
   &a^{2m-1}_{i_1\cdots i_{2m-1}} (e_1+\I\eta)_{i_1}\cdots (e_1+\I\eta)_{i_{2m-1}}\nonumber\\
   &= \sum\limits_{p=0}^{2m-1} a_{i_1\cdots i_p}^{2m-1,p}\, \eta_{i_1} \cdots \eta_{i_p}, \quad 2\le i_1,\cdots,i_p\le n \quad \mbox{for each} \quad 0\le p \le 2m-1.
\end{align*}%
where $ a_{i_1\cdots i_p}^{2m-1,p} = (\I)^p\, \binom{2m-1}{p}\, a_{i_1\cdots i_p1\cdots 1}^{2m-1}= c_{p}\, (\I)^{p}\, a_{i_1\cdots i_p1\cdots 1}^{2m-1} $.
 This together with \eqref{main_identity} implies
 \begin{align*}
     \int\limits_{\Rn}\sum\limits_{p=0}^{2m-1} a_{i_1\cdots i_p}^{2m-1,p}\, \eta_{i_1} \cdots \eta_{i_p} a_0\, b_0=0. 
 \end{align*}
 Next we choose $a_0= y_2^{m-1}\,g(y'') e^{-\I\lambda(y_1+\I y_2)},\quad b_0= y_2^{m-1}$. Then utilizing the arbitrariness of $g$ and taking partial Fourier transform in the first variable, we obtain
 \begin{align}\label{eq_5.8}
      \int\limits_{\R} e^{\lambda y_2}\, \sum\limits_{p=0}^{2m-1} \widehat{a}_{i_1\cdots i_p}^{2m-1,p}(\lambda,y_2,y'')\, \eta_{i_1} \cdots \eta_{i_p}\, y^{2m-2}_2\, dy_2=0.
 \end{align}
 Next, we specify  $\lambda=0$ in above, after that we  replace $ \eta$ by $-\eta$  and again put $\lambda=0$ in above. As a result, we obtain following two equations.
 \begin{align*}
 0=\begin{cases}
     \int\limits_{\R}  \sum\limits_{p=0}^{2m-1} \widehat{a}_{i_1\cdots i_p}^{2m-1,p}(0,y_2,y'')\, \eta_{i_1} \cdots \eta_{i_p}\, y^{2m-2}_2\, dy_2\\
      \int\limits_{\R}  \sum\limits_{p=0}^{2m-1} (-1)^p\, \widehat{a}_{i_1\cdots i_p}^{2m-1,p}(0,y_2,y'')\, \eta_{i_1} \cdots \eta_{i_p}\, y^{2m-2}_2\, dy_2.
       \end{cases}
 \end{align*}
 We add and subtract above two relations. This entails
 \begin{align*}
     \int\limits_{\R}  \sum\limits_{p=0}^{2m-1} (1\pm (-1)^p)\, \widehat{a}_{i_1\cdots i_p}^{2m-1,p}(0,y_2,y'')\, \eta_{i_1} \cdots \eta_{i_p}\, y^{2m-2}_2\, dy_2=0. 
 \end{align*}
 This allows us to separate even and odd modes of $a^{2m-1}$. After a re-indexing this entails
\begin{align*}
0=\begin{cases}
      \int\limits_{\R}  \sum\limits_{p=0}^{m-1} \widehat{a}_{i_1\cdots i_{2p+1}}^{2m-1,2p+1}(0,y_2,y'')\, \eta_{i_1} \cdots \eta_{i_p}\, y^{2m-2}_2\, dy_2 \\
   \int\limits_{\R}  \sum\limits_{p=0}^{m-1} \widehat{a}_{i_1\cdots i_{2p}}^{2m-1,2p}(0,y_2,y'')\, \eta_{i_1} \cdots \eta_{i_p}\, y^{2m-2}_2\, dy_2.
\end{cases}
\end{align*}
Utilizing   \Cref{increasing_order}, the next identities follow from the  definition of $ a^{2m-1,odd}$ and $ a^{2m-1,even}$.
\begin{align}
     \int\limits_{\R}   \widehat{a}_{i_1\cdots i_{2m-1}}^{2m-1,odd}(0,y_2,y'')\, \eta_{i_1} \cdots \eta_{i_{2m-1}}\, y^{2m-2}_2\, dy_2=0, \label{mrt_of_odd_mode}\\
    \int\limits_{\R}   \widehat{a}_{i_1\cdots i_{2m-2}}^{2m-1,even}(0,y_2,y'')\, \eta_{i_1} \cdots \eta_{i_{2m-2}}\, y^{2m-2}_2\, dy_2=0. \label{mrt_of_even_mode}
\end{align}
\smallskip
\textbf{Step 3.}  We now assume the additional conditions \eqref{div_trace_free-assumption}. Then we claim that 
\begin{align*}
     \widehat{a}^{2m-1}(0,x')=0. 
     \end{align*}\smallskip
Observe that, combining  Lemma \ref{translation_lemma} and \eqref{mrt_of_odd_mode} we have $   \int\limits_{\R}   \widehat{a}_{i_1\cdots i_{2m-1}}^{2m-1,odd}(0,y_2,y'')\, \eta_{i_1} \cdots \eta_{i_{2m-1}}\, y^{p}_2\, dy_2=0 $  for all $0\le p \le 2m-2$. In other words, vanishing the highest order moment is enough.  This along with  \cite[Theorem 2.17.2]{Sharafutdinov_book} entails that there is a distribution $\phi(0,x')$ satisfying 
 $ \widehat{a}^{2m-1,odd}(0,x')= d_{x'}^{2m-1}\phi(0,x').
 $
 This implies 
  \begin{align}\label{eq_5.19}
      \widehat{a}^{2m-1,2m-1} = -\sum\limits_{p=0}^{m-2} i'^{m-1-p}_{\d}\widehat{a}^{2m-1,2p+1} +d_{x'}^{2m-1}\phi.
 \end{align}
From \cite[Theorem 2.17.2]{Sharafutdinov_book} we also have 
 $
     \phi(0,x')=0$ outside of some ball in $\R^{n-1}$.

Using the fact that $\widehat{a}^{2m-1,odd}(0,x') \in L^2(\R^{n-1})$ and  $d_{x'}^{2m-1}\phi(0,x') = (\partial_{j_1 \ldots j_{2m-1}} \phi(0,x'))_{2 \leq j_1, \ldots, j_{2m-1} \leq n}$, it follows that $ \phi\in H^{2m-1}(\mathbb{R}^{n-1})$. Here $H^k(\mathbb{R}^{n-1})$ denotes the $L^2$ based Sobolev space of order $k$; see  e.g.\ \cite[Chapter 5]{Evans_pde_book} for precise definition.
 We next consider the  $L^2$ inner product of $ \widehat{a}^{2m-1,2m-1}$ in the following way:
 \begin{align*}
    \langle   \widehat{a}^{2m-1,2m-1}(0,x'),  \widehat{a}^{2m-1,2m-1}(0,x') \rangle =  \sum\limits_{i_1,\cdots i_{2m-1}=2}^{n} \int_{\mathbb{R}^{n-1}} \widehat{a}_{i_1\cdots i_{2m-1}}^{2m-1,2m-1} (0,x') \, \widehat{a}_{i_1\cdots i_{2m-1}}^{2m-1,2m-1} (0,x') \,d x'.
 \end{align*}
 This together with \eqref{eq_5.19} entails
 \begin{equation}\label{eq_5.21}
      \begin{aligned}
    & \langle   \widehat{a}^{2m-1,2m-1},  \widehat{a}^{2m-1,2m-1} \rangle\\ & \quad=  \left\langle   \widehat{a}^{2m-1,2m-1}, -\sum\limits_{p=0}^{m-2} i'^{m-1-p}_{\d}\widehat{a}^{2m-1,2p+1} +d_{x'}^{2m-1}\phi   \right\rangle\\
     &\quad= \left\langle   \widehat{a}^{2m-1,2m-1}, -\sum\limits_{p=0}^{m-2} i'^{m-1-p}_{\d}\widehat{a}^{2m-1,2p+1} \right \rangle  + \left\langle \widehat{a}^{2m-1,2m-1}, d_{x'}^{2m-1}\phi   \right\rangle\\
     &\quad= - \sum\limits_{p=0}^{m-2}\left\langle   \widehat{a}^{2m-1,2m-1},  i'^{m-1-p}_{\d}\widehat{a}^{2m-1,2p+1} \right \rangle  + \left\langle \widehat{a}^{2m-1,2m-1}, d_{x'}^{2m-1}\phi   \right\rangle.
 \end{aligned}
 \end{equation}
Now we  take the Fourier transform of $ j_{\d} a^{2m-1}=0$ in the $x_1$ variable and evaluate at the origin  to obtain
$ j_{\d} \widehat{a}^{2m-1}(0,x')=0. 
$
This implies 
\begin{align*}
     j_{\d}'\widehat{a}^{2m-1,2m-1} &= \sum\limits_{k=2}^n \widehat{a}_{i_1\cdots i_{2m-3}kk}^{2m-1,2m-1} = \sum\limits_{k=2}^n (\I)^{2m-1}\,\widehat{a}_{i_1\cdots i_{2m-3}kk}^{2m-1}= - (\I)^{2m-1} \widehat{a}_{i_1\cdots i_{2m-3}11}^{2m-1}\\
     &= -\,\binom{2m-1}{2m-3}^{-1}  (\I)^{2m-1-(2m-3)} \widehat{a}^{2m-1,2m-3}= \frac{1}{c_{2m-3}} \widehat{a}^{2m-1,2m-3}.
\end{align*}
Proceeding in this way we obtain
\begin{align}\label{lambda_zero_trace}
  j'^{m-1-p}_{\d}\widehat{a}^{2m-1,2m-1} &= \frac{1}{c_p}\,\widehat{a}^{2m-1,2p+1} \quad \mbox{for} \quad 0\le p\le m-2.   
\end{align}
Taking Fourier transform of the relation  $ \delta^{2m-1} a^{2m-1}=0$ with respect to $x_1$ variable gives
$
   \sum\limits_{k=0}^{2m-1}\binom{2m-1}{k}\, \lambda^{k} \, \delta_{x'}^{2m-1-k}  \widehat{a}^{2m-1}(\lambda,x') =0.
$ %
This for $\lambda=0$ entails
\begin{align*}
   \delta_{x'}^{2m-1}  \widehat{a}^{2m-1}(0,x') =0\implies \delta_{x'}^{2m-1}  \widehat{a}^{2m-1,2m-1}(0,x') =0.
\end{align*}%
Let $ \phi_N\in C_c^{\infty}(\mathbb{R}^{n-1})$ such that $ \phi_N \rightarrow \phi $ in $ H^{2m-1}$, then from above as $N \rightarrow \infty $ we see that \[ 0=(-1)^{2m-1}\langle \delta_{x'}^{2m-1}  \widehat{a}^{2m-1,2m-1},\phi_N\rangle =  \,\langle \widehat{a}^{2m-1,2m-1},d_{x'}^{2m-1}\phi_N \rangle \rightarrow \langle \widehat{a}^{2m-1,2m-1},d_{x'}^{2m-1}\phi \rangle.\] This implies $ \langle \widehat{a}^{2m-1,2m-1},d_{x'}^{2m-1}\phi \rangle=0$.
The combination of this along with \eqref{eq_5.21} and  \eqref{lambda_zero_trace}  implies
\begin{align*}
   \langle   \widehat{a}^{2m-1,2m-1},  \widehat{a}^{2m-1,2m-1} \rangle= - \frac{1}{c_p} \sum\limits_{p=0}^{m-2}\left\langle    \widehat{a}^{2m-1,2p+1}, \widehat{a}^{2m-1,2p+1} \right \rangle.
   \end{align*}
   To put it another way, it gives
   \begin{align*}
        \langle   \widehat{a}^{2m-1,2m-1},  \widehat{a}^{2m-1,2m-1} \rangle+\sum\limits_{p=0}^{m-2} \frac{1}{c_p} \left\langle    \widehat{a}^{2m-1,2p+1}, \widehat{a}^{2m-1,2p+1} \right \rangle=0.
   \end{align*}
  This in addition to the fact that  $c_p>0$ for $0\le p\le m-2$ implies  \begin{align*}
     \widehat{a}^{2m-1,2p+1}(0,x')=0 \quad \mbox{for} \quad 0\le p\le m-1.  \end{align*}%
     By  Lemma \ref{kernel_mrt_sphere_bundle} and $ j_{\delta} a^{2m-1}=0 $ from \eqref{mrt_of_even_mode} we obtain
    $
     \widehat{a}^{2m-1,even}(0,x')=0.
$
   This together with $j_{\d} a^{2m-1}=0$ implies
   $
\sum\limits_{p=0}^{m-1} \frac{1}{c_p} \,\left\langle \widehat{a}^{2m-1,2p}, \widehat{a}^{2m-1,2p} \right \rangle=0.
   $ We have that $c_p>0$  for all $p$ with $0\le p\le m-1$, this entails  
   $
     \widehat{a}^{2m-1,2p}(0,x')=0$ for all $ 0\le p\le m-1.$ 
  The combination of $ \widehat{a}^{2m-1,2p}(0,x')=  \widehat{a}^{2m-1,2p+1}(0,x')=0$ for all $p$ with $ 0\le p \le m-1$  implies
 $ \widehat{a}^{2m-1}(0,x')=0.  
  $ 
  \smallskip
  
  \textbf{Step 4.} We still assume the additional conditions in  \eqref{div_trace_free-assumption}, and show that $ a^{2m-1}=0 $.\\
  
  Since $ a^{2m-1}$ is compactly supported in the $x_1$  variable, this implies $ \widehat{a}^{2m-1}(\lambda,x')$ is analytic in $ \lambda$ by Paley-Wiener theorem. Therefore it is enough to show that 
\begin{align*}
      \frac{\D^l}{ \D\lambda^l} \widehat{a}^{2m-1}(0,x')=0\quad \mbox{ for all non-negative integers} \quad l.
  \end{align*}
  From previous step  we have $\widehat{a}^{2m-1}(0,x')=0 $. 
 We now argue by induction  and assume that \begin{align}\label{induction_hypo}
      \frac{\D^l}{ \D\lambda^l} \widehat{a}^{2m-1}(0,x')=0 \quad \mbox{ for all integers} \quad l \quad \mbox{with} \quad 0< l \le M.
  \end{align}  For $M+1$, differentiating \eqref{eq_5.8} $M+1$ times with respect to $\lambda$  gives
  \begin{align*}
  &    \sum\limits_{k=0}^{M+1} \int\limits_{\R} \binom{M+1}{k} \, y_2^{M+1-k} e^{\lambda y_2}\,   \frac{\D^{k}}{\D \lambda^k}\left(\sum\limits_{p=0}^{2m-1} \widehat{a}_{i_1\cdots i_p}^{2m-1,p}(\lambda,y_2,y'') \right)\, \eta_{i_1} \cdots \eta_{i_p}\, y^{2m-2}_2\, dy_2=0.
  \end{align*}
  We now specify $\lambda = 0$ in preceding equation and then utilize \eqref{induction_hypo} to conclude
   \begin{align}\label{M+1_derivative}
   \int\limits_{\R}      \sum\limits_{p=0}^{2m-1}\left(\frac{\D^{M+1}}{\D \lambda^{M+1}} \widehat{a}_{i_1\cdots i_p}^{2m-1,p}(0,y_2,y'') \right)\, \eta_{i_1} \cdots \eta_{i_p}\, y^{2m-2}_2\, dy_2=0.   
   \end{align}
  Differentiating  the relation $\sum\limits_{k=0}^{2m-1}\binom{2m-1}{k}\, \lambda^{k} \, \delta_{x'}^{2m-1-k}  \widehat{a}^{2m-1}(\lambda,x') =0$, $M+1$ times more with respect to $\lambda$ gives
  \begin{align}\label{eq_5.29}
      \sum\limits_{r=0}^{M+1} \sum\limits_{k=r}^{2m-1}\binom{2m-1}{k}\, \frac{k!}{(k-r)!}\lambda^{k-r} \, \frac{\D^{M+1-r}}{\D \lambda^{M+1-r}} \delta_{x'}^{2m-1-k}  \widehat{a}^{2m-1}(\lambda,x') =0.
  \end{align}
  From  \eqref{induction_hypo} we obtain 
  \begin{align*}
      \frac{\D^{M+1-r}}{\D \lambda^{M+1-r}} \delta_{x'}^{2m-1-k}  \widehat{a}^{2m-1}(0,x') =0  \quad \mbox{for} \quad 1\le r \le M+1 \quad \mbox{and} \quad 1\le k \le 2m-1.
  \end{align*}
  In the last relation above  we have used the fact that, $\frac{\D }{\D\lambda}$ commutes with $ \PD_{x'}$.  Setting $\lambda=0$ in \eqref{eq_5.29}, then utilizing above findings   we obtain  
$
    \frac{\D^{M+1}}{\D \lambda^{M+1}} \delta_{x'}^{2m-1}  \widehat{a}^{2m-1}(0,x') =0.
 $
  Since $j_{\d} a^{2m-1}=0$, this implies
   $
       j_{\d} \left(\frac{\D^{M+1}}{\D \lambda^{M+1}} \widehat{a}^{2m-1}(0,x')\right)=0.
   $
 Using the last two relations 
 and repeating the similar analysis as before for the integral \eqref{M+1_derivative} we obtain
 $ \frac{\D^{M+1}}{\D \lambda^{M+1}}  \widehat{a}^{2m-1}(0,x')=0.
  $
  This completes the induction step and implies
  \begin{align*}
      \frac{\D^l}{ \D\lambda^l} \widehat{a}^{2m-1}(0,x')=0\quad \mbox{ for all non-negative integers} \quad l.
  \end{align*}
Combining  this with Paley-Wiener theorem we obtain
  $ a^{2m-1}(x)=0 \quad\mbox{in}\quad \Omega.
  $
  In combining all of the above steps with Theorem \ref{th:main_theorem_1}, we conclude that $
      a^j(x)=0 \quad \mbox{for all} \quad 0\le j\le 2m-1.
  $
 To complete the proof of Proposition \ref{main_theorem_2}  we will use a decomposition result proved in Lemma \ref{trace_free_helmholtz decomposition}.\smallskip  
   
  \textbf{Step 5.} In this step we prove \eqref{highest_order_tensor}.\\


\noindent We first substitute  $u$ and $v$ from \eqref{cgos2} in the integral identity \eqref{main_identity_4} and obtain
 \begin{align*}
     \int\limits_{\Omega} \sum_{l=0}^{2m-1}a^{l}_{i_1\cdots i_l} D^{i_1\cdots i_l} \left(e^{\frac{(e_1 + \I \eta)\cdot x}{h}} \tilde{A} \right) \,  e^{-\frac{(e_1 + \I \eta)\cdot x}{h}} \tilde{B}=0.
 \end{align*}
 Multiplying above by $h^{2m-1}$  and letting $h\rightarrow 0$ we arrive at
 $ \int_{\Omega} a^{2m-1}_{i_1\cdots i_{2m-1}}\,(e_1+\I\eta)_{i_1}\cdots (e_1+\I\eta)_{i_{2m-1}} a_0\, b_0=0,$
 where $a_0 $ and $b_0$ solve the  transport equation
$     T^ma_0=T^mb_0=0.$
Thanks to Lemma \ref{trace_free_helmholtz decomposition} we can decompose $a^{2m-1}$ in the following way:
 \begin{align*}
  a^{2m-1}= \tilde{a}^{2m-1}+ i_{\delta} b^{2m-3}+\D^{2m-1}\phi\quad \mbox{with} \quad \PD_{\nu}^l\phi|_{\PD\Omega}=0 \quad \mbox{for}\quad l=0,1,\cdots,2m-2,
 \end{align*}%
 and $\tilde{a}^{2m-1}$ meets the following conditions
  $ \delta^{2m-1} \tilde{a}^{2m-1}=j_{\delta} \tilde{a}^{2m-1}=0 $ in $\Omega.$ This is known as \emph{trace free Helmholtz type decomposition} of symmetric  tensor fields proved in Section \ref{sec:decomposition}. Plugging this decomposition of $a^{2m-1}$ into $ \int_{\Omega} a^{2m-1}_{i_1\cdots i_{2m-1}}\,(e_1+\I\eta)_{i_1}\cdots (e_1+\I\eta)_{i_{2m-1}} a_0\, b_0=0$  and utilizing  $ \langle e_1+\I\eta,  e_1+\I \eta\rangle=0$  we obtain
 \begin{align}\label{eq_5.36}
      \int\limits_{\Omega} (\tilde{a}^{2m-1} +\D^{2m-1}\phi)_{i_1\cdots i_{2m-1}}  (e_1+\I\eta)_{i_1}\cdots (e_1+\I\eta)_{i_{2m-1}} a_0\, b_0=0. 
 \end{align}
 Since $ \phi$ fulfils $\PD_{\nu}^l\phi|_{\PD\Omega}=0$ for $ 0\le l\le 2m-2$, this along with integration by parts  entails
 \begin{align*}
    (-1)^{2m-1} \int\limits_{\Omega} (\D^{2m-1}\phi)_{i_1\cdots i_{2m-1}}  (e_1+\I\eta)_{i_1}\cdots (e_1+\I\eta)_{i_{2m-1}} a_0\, b_0
   =\int\limits_{\Omega}\phi\,  \sum\limits_{k=0}^{2m-1}\,\binom{2m-1}{k} T^{2m-1-k} a_0\, T^k b_0.
    \end{align*}
We next make use of transport equations $ T^ma_0=0$ and $T^mb_0=0$ to deduce $\int_{\Omega}\phi\,  \sum_{k=0}^{2m-1}\,\binom{2m-1}{k} T^{2m-1-k} a_0\, T^k b_0$. This implies $ \int_{\Omega} (\D^{2m-1}\phi)_{i_1\cdots i_{2m-1}}  (e_1+\I\eta)_{i_1}\cdots (e_1+\I\eta)_{i_{2m-1}} a_0\, b_0    =0.$
 Combining this with \eqref{eq_5.36} we arrive at 
 $ \int_{\Omega} \tilde{a}^{2m-1}_{i_1\cdots i_{2m-1}}  (e_1+\I\eta)_{i_1}\cdots (e_1+\I\eta)_{i_{2m-1}} a_0\, b_0=0.$   
 We next extend $ \tilde{a}^{2m-1}$ to be $0$ outside of $ \overline{\Omega}$ and obtain
 \begin{align*}
       \int\limits_{\Rn} \tilde{a}^{2m-1}_{i_1\cdots i_{2m-1}}  (e_1+\I\eta)_{i_1}\cdots (e_1+\I\eta)_{i_{2m-1}} a_0\, b_0=0, \quad \mbox{where $a_0$ and $b_0$ solve $T^m (\cdot)=0$}.
 \end{align*}
Since $\tilde{a}^{2m-1} $ satisfies the conditions in \eqref{div_trace_free-assumption}, utilizing previous steps  we obtain $ \tilde{a}^{2m-1}=0$.
 This implies $ a^{2m-1}= i_{\delta} b^{2m-3}+\D^{2m-1}\phi$, where $\phi$ satisfies $\PD_{\nu}^l\phi|_{\PD\Omega}=0$ for $ 0\le l\le 2m-2$.
 This completes the proof. 
  \end{proof}
    \section{Proof of theorem \ref{main_theorem_3}}\label{Gauge transformation}
    The proof of Theorem \ref{main_theorem_3} will be presented in this section. However, before we get to the proof of Theorem \ref{main_theorem_3}, let us have a look at the gauge transformation for the case of  $m=2$.
    \subsection{Gauge transformation} To this end we denote
\begin{align*}
    \Lc_{a^3,a^2,a^1,a^0} &= (-\Delta)^2  + a_{ijk}^3 D^{ijk}+ a_{ij}^2 D^{ij}+a_{i}^1 D^{i}+a^0, \\\
    \Lc_{0} &= (-\Delta)^2.
\end{align*}
Note that Cauchy data of $v$ and  $ e^{\phi} v$ up to order $3$ are the same if $ \phi \in C^4(\overline{\Omega})$ satisfies $\PD_{\nu}^j \phi|_{\partial \Omega}=0$ for $0 \leq j \leq 3$. Now we compute 
\begin{align*}
   e^{-\phi} \Lc_0 e^{\phi} v  &=  e^{-\phi} (-\Delta)^2 e^{\phi} v\\
   &=  (e^{-\phi} (-\Delta) e^{\phi})^2v \\
   &= (\Delta\phi + |\nabla \phi|^2+ 2 \nabla\phi\cdot \nabla +\Delta ) (\Delta\phi + |\nabla \phi|^2+ 2\nabla\phi\cdot \nabla +\Delta)v\\
  &= (\Delta )^2v +\Delta (\nabla\phi\cdot \nabla  v) + \Delta (f(\phi) v)+ 2 \nabla\phi\cdot \nabla (\Delta)v \\  & \qquad + 4\nabla\phi\cdot \nabla (\nabla\phi\cdot \nabla v) + 2 \nabla\phi\cdot \nabla (f(\phi) v)  +f^2(\phi) v+ 2f(\phi)\,  \nabla\phi\cdot \nabla v +f(\phi) \Delta v\\
   &=(\Delta)^2v +\underbrace{4 \nabla\phi\cdot \nabla \Delta v}_{\mbox{\tiny{third order term}}} + 4\nabla^2  \phi \cdot \nabla^2 +4 \nabla\phi \otimes \nabla \phi \cdot \nabla^2 + 2 f(\phi) \,\Delta \\& \qquad + 2\nabla \Delta \phi\cdot \nabla v+ 4\nabla\phi \cdot \nabla^2\phi\cdot \nabla v+ f(\phi) \nabla\phi\cdot\nabla v+2 \nabla f(\phi)\cdot\nabla v+ 4 f(\phi) \nabla \phi\cdot\nabla v  \\& \qquad+ v\Delta (f(\phi)) +2\,v  \nabla \phi \cdot \nabla f(\phi)  + v\,f^2(\phi),
\end{align*}
where $ f(\phi)= \Delta \phi +|\nabla \phi|^2$.

 The above calculation shows that if $\phi$ is as above, then 
\[
C_{a^3,a^2,a^1,a^0}=C_{0},
\]
where $a^3 = 4 i_{\delta} (\nabla \phi) $, and $a^2, a^1, a^0$ can be read off from the computation above. Such calculations can also be done for $m\ge 3$, but it will be somewhat cumbersome.
\subsection{Summary of results}
 Here we summarize few important results we have gotten so far from previous sections since  they will help us and  the reader when we apply them in this section.
  \begin{lemma}\label{summary_prop_4.3}
  Let $a^M \in C^{\infty}(\overline{\Omega})$ be a symmetric $M$ tensor field where $M\ge 0$ is an integer. 
  \begin{itemize}
  \item [1.] Let $M=2m-1$ and  assume that 
  \begin{align*}
       \int\limits_{\Rn} a^{2m-1}_{i_1\cdots i_{2m-1}} (e_1+\I\eta)_{i_1}\cdots (e_1+\I\eta)_{i_{2m-1}} a_0\, b_0=0 \quad \mbox{for all unit vectors $\eta \perp e_1$} 
  \end{align*}
  where  $ T^{m}a_0= T^{m}b_0=0$. Then we have  $a^{2m-1}=d^{2m-1}\phi +\id a^{2m-1,1} $, where $a^{2m-1,1} $ is a smooth symmetric tensor field of order $2m-3$, and $\phi \in C^{\infty}(\overline{\Omega})$  satisfies $\PD^l_{\nu} \phi|_{\PD\Omega}=0$ for $0\le l\le 2m-2$.
  \item [2.] Let $M=2m$ and  assume that 
  \begin{align*}
       \int\limits_{\Rn} a^{2m}_{i_1\cdots i_{2m}} (e_1+\I\eta)_{i_1}\cdots (e_1+\I\eta)_{i_{2m}} T a_0\, b_0=0 \quad \mbox{for all unit vectors $\eta \perp e_1$} 
  \end{align*}
  where  $ T^{m+1}a_0= T^{m+1}b_0=0$. Then we have  $a^{2m}=d^{2m}\phi_1 +\id a^{2m,1} $, where $a^{2m,1} $ is a smooth symmetric tensor field of order $2m-2$, and $\phi_1 \in C^{\infty}(\overline{\Omega})$  satisfies $\PD^l_{\nu} \phi_1|_{\PD\Omega}=0$ for $0\le l\le 2m-1$.
  \end{itemize}
  \end{lemma}
  \begin{proof}
    For first part we refer \textbf{Step 2-5} 
     in the proof of \Cref{main_theorem_2}. One can employ  similar line of arguments for second part as well.
  \end{proof}
  \begin{lemma}\label{lm:summary_1}
  Let $ a^l$ be the same as above such that $l\le m$.
      Suppose there holds 
  \begin{align*}
      \sum\limits_{j=1}^{m} \int\limits_{\Omega} a^{j}_{i_1\cdots i_j}\, (e_1+\I \eta)_{i_1} \cdots (e_1+\I \eta)_{i_j}\, T^{m-j}a_0b_0=0 \quad \mbox{for all unit vectors $\eta \perp e_1$}  
  \end{align*}
 where  $ T^{m}a_0= T^{m}b_0=0$. Then for some tensor $a^{j,1}$ of order $j-2$ we have 
 \begin{align*}
     a^{j} = d^{j}\phi_j +\id a^{j,1} \quad \mbox{where $\phi_j$ matches the boundary condition $ \PD^{l}\phi_j|_{\PD\Omega}=0$ for $ 0\le l\le j-1$}.
 \end{align*}
 Note that,    $a^{1,1}=0$.
  
  \end{lemma}
  \begin{proof}
    The proof relies on the careful choice of solution of transport equations $a_0$ and $b_0$ along with second condition of \Cref{lem:sum_of_mrt} and Lemma \ref{summary_prop_4.3}. The proof of \Cref{th:main_theorem_1} which uses the first hypothesis of \Cref{lem:sum_of_mrt}, involves a similar argument; see in  \textbf{Step 2} from equation \eqref{eq_3.14}-\eqref{eq_3.20}.
  \end{proof}

\subsection{Proof of \Cref{main_theorem_3}}
\begin{proof}
The proof of \Cref{main_theorem_3} will be divided into several steps. Note that, for simplicity, in the proof we may use same symbols such as $\phi_{0},\phi_{1},$ to express tensor fields of different order. \smallskip

\textbf{Step 1.} The case of $m=2$.\smallskip

By the assumption of Theorem \ref{main_theorem_3} we have
 \begin{align*}
	        \int\limits_{\Omega} \sum_{l=0}^{3}a^{l}_{i_1\cdots i_l} D^{i_1\cdots i_l}u\, v=0\quad \mbox{whenever} \quad  \Delta^2 u= \Delta^2 v=0. 
	    \end{align*}
We insert CGO solutions of  $u$ and $ v$ given in \eqref{cgos2} and derive
\begin{align}\label{main_identity_3}
     \int\limits_{\Omega} \sum_{l=0}^{3}a^{l}_{i_1\cdots i_l} D^{i_1\cdots i_l} \left(e^{\frac{(e_1 + \I \eta)\cdot x}{h}} \tilde{A} \right) \, e^{-\frac{(e_1 + \I \eta)\cdot x}{h}} \tilde{B}=0.
\end{align}
To continue we then  multiply \eqref{main_identity_3} by $h^{3}$ and  let $ h\rightarrow 0$ to obtain
\begin{align*}
    \int\limits_{\Omega} a_{i_1\cdots i_3}^{3} (e_1+\I \eta)_{i_1}\cdots (e_1+\I \eta)_{i_3}\, a_0\,b_0=0.
\end{align*}
Here $ a_0$ and $b_0$ solve $T^2a_0=T^2b_0=0$. This together with Proposition \ref{main_theorem_2} entails
\begin{align}\label{eq_6.3}
    a^3= \D^3 \phi_0 +i_{\d} a^{3,1}\quad \mbox{where $ \phi_0$ fulfils} \quad \PD^{l}_{\nu}\phi_0|_{\PD\Omega}=0\quad \mbox{for} \quad 0\le l \le 2
\end{align}
where  $ a^{3,1}$ is a smooth vector field in $ \overline{\Omega}$. Recall that, $a^{m,k}$ denotes a symmetric  tensor field of order $m-2k$ for every non-negative integers $m$ and $k$ for which $m-2k\ge 0$. After that, we see that the coefficient of $h^3$ is indeed zero by plugging \eqref{eq_6.3} into \eqref{main_identity_3}. However, it can be seen from the discussion at the beginning of this section that we have not yet obtained the exact gauge, which is $ i_{\d} (\nabla \phi)$, for some scalar $ \phi$. This differs from  inverse problems involving second order elliptic partial differential equations, where the gauge can be obtained in a single step by multiplying a particular integral equation by $h$ and then letting $h\rightarrow 0$. The form of $a^3$ presented in \eqref{eq_6.3}  can eliminate the coefficient of $h^3$, but not the one corresponding to $h^2$. As a result, multiplying (5.1) by $h^2$ and setting $h\rightarrow 0$ yields
    \begin{align*}
       3 \int\limits_{\Omega} a_{i_1\cdots i_3}^{3} (e_1+\I \eta)_{i_1} (e_1+\I \eta)_{i_2} \, \PD_{x_{i_3}}a_0\, b_0  +  \int\limits_{\Omega} a_{i_1 i_2}^2\, (e_1+\I \eta)_{i_1}\,  (e_1+\I \eta)_{i_2} a_0\, b_0=0.
        \end{align*}
       This together with \eqref{eq_6.3}   implies
        \begin{align}\label{eq_6.4}
     0=  &  3 \int\limits_{\Omega} \PD^3_{i_1 i_2 i_3} \phi_0\,(e_1+\I \eta)_{i_1} (e_1+\I \eta)_{i_2} \, \PD_{x_{i_3}}a_0\, b_0  +  \int\limits_{\Omega} a_{i_1 i_2}^2\, (e_1+\I \eta)_{i_1}\,  (e_1+\I \eta)_{i_2} a_0\, b_0\nonumber\\& \quad + \int\limits_{\Omega} a^{3,1}_i\, (e_1+\I \eta)_i \,T a_0\, b_0.
    \end{align}
   Combining   integration by parts, boundary conditions of $\phi_0$ and $ T^2a_0=T^2b_0=0$ from above we derive 
        \begin{align*}
    \frac{3}{4}\,\int\limits_{\Omega} \PD_{ i_3} \phi_0\, \PD_{x_{i_3}}Ta_0\, Tb_0  +  \int\limits_{\Omega} a_{i_1 i_2}^2\, (e_1+\I \eta)_{i_1}\,  (e_1+\I \eta)_{i_2} a_0\, b_0+ \int\limits_{\Omega} a^{3,1}_i\, (e_1+\I \eta)_i \,T a_0\, b_0&=0.
        \end{align*}
        The factor $\frac{1}{4}$ appears in the above expression because $ T= 2 (e_1+\I \eta)\cdot \PD_x$.
        Now we choose $ a_0 =x \cdot \eta$. This gives $ Ta_0 =2\I$ and $ \PD_{x_{i_3}} Ta_0=0$ for $1\le i_3\le n$.
        This implies
        \begin{align*}
         \int\limits_{\Omega} a_{i_1 i_2}^2\, (e_1+\I \eta)_{i_1}\,  (e_1+\I \eta)_{i_2} (x \cdot \eta)\, b_0+ 2\I\, \int\limits_{\Omega} a^{3,1}_i\, (e_1+\I \eta)_i \, b_0&=0.   
        \end{align*}
       We next use Lemma \ref{lm:summary_1} for $m=2$ to obtain $ a^2 =\D^2\phi_1 +i_{\d} a^{2,1}$ and $b^1= \D \phi_2$.  Substituting this into above integral identity and using the boundary conditions of $\phi_1$ and $\phi_2$ we arrive at 
         $$ a^{3,1}= \D \phi_1 \quad \mbox{and} \quad a^2 =\D^2\phi_1 +i_{\d} a^{2,1}$$
         where $ \phi_1$ satisfies $ \PD^l_{\nu}\phi_1|_{\PD\Omega}=0$ for $ 0\le l\le 1$.
         Inserting  the above form of $ a^{3,1}$ and $a^2$  in \eqref{eq_6.4}   and then performing integration by parts we obtain
       $
            \int_{\Omega}  \phi_0\,\left(\, \Delta Ta_0\, Tb_0+ \left\langle \nabla T  a_0,\nabla T b_0\right\rangle\,\right)=0.$
        Extending $ \phi_0=0$ outside $ \overline{\Omega}$  and doing the change of variable $ x\mapsto y$ we obtain
        $
           \int_{\Rn}\phi_0\,\big( \Delta Ta_0\, Tb_0+ \langle \nabla T  a_0,\nabla Tb_0 \rangle\big)  =0.
       $%
       We choose $ Ta_0= y_2$ and $ Tb_0= g(y'') \, e^{-\I \lambda (y_1+\I y_2)}$ for all $\lambda$ with $\lambda \neq 0$ and $g$ is any smooth function in $y''$ variable. This implies \[ \Delta Ta_0=0, \quad \nabla Ta_0= (0,1,0,\cdots,0)\quad \mbox{and} \quad \nabla Tb_0= e^{-\I \lambda (y_1+\I y_2)}\lr{-\lambda\I\, g(y''),\lambda g(y''),\nabla_{y''} g(y'')} .  \]
       Inserting above into the integral identity $ \int_{\Rn}  \phi_0\,\lr{\, \Delta Ta_0\, Tb_0+ \langle \nabla T  a_0,\nabla Tb_0 \rangle  }=0 $ and then varying $g$ and taking partial Fourier transform in the first variable, we obtain for almost every $y''$ that
           $ \int_{\mathbb{R}}  \widehat{\phi}_0(\lambda,y_2,y'')\, e^{\lambda y_2} \, dy_2=0$ for all $ \lambda$ with $\lambda\neq 0.$ 
Since $y_2$ varies in a compact set,  one can let $\lambda \rightarrow 0$ using Lebesgue dominated convergence theorem  and  obtain
   $   \int_{\mathbb{R}}  \widehat{\phi}_0(0,y_2,y'')\, \, dy_2=0.$
 This along with uniqueness of the  ray transform \cite[Chapter 2]{Sharafutdinov_book} implies   $ \phi_0(0,x')=0$. One can now repeat arguments used before to show that  $ \frac{\D^l}{\D \lambda^l} \phi_0(\lambda,x')|_{\lambda=0}=0$ for all non-negative integers $l$. This implies $\phi_0=0$ by Paley-Wiener. As a result we obtain
\begin{align*}
\begin{cases}
    a^3= i_{\d} (\nabla \phi_1) \qquad \hspace{6.8mm} \mbox{with} \quad \phi_1|_{\PD\Omega}=0\\
    a^2=\frac{1}{2}\D^2\phi_1 +i_{\d} a^{2,1}  \quad \mbox{with} \quad  \phi_1|_{\PD\Omega}= \PD_{\nu}\phi_1|_{\PD\Omega} =0.
    \end{cases}
\end{align*}
Since $ \PD_{\nu}^{l} a^3|_{\PD\Omega}=0$ for $ 0\le l\le 3$, this implies $\PD_{\nu}^{l} \phi_1=0 $ on $\PD \Omega$ for $0\le  l\le 3$. Thus $ \phi_1$ satisfies the  conditions in \eqref{phi_conditions}.  
We now perform the gauge transformation in the third order perturbation and replace $a^3$ by $0$.  After gauge transformation  lower order terms will change accordingly. We denote them by $\tilde{a}^{l} $ for $l=0,1,2$. Note that $ \tilde{a}^l$'s are now depends on $ \phi_1$ and can be recovered using Theorem \ref{th:main_theorem_1}. In other words we obtain 
\begin{align*}
    \Lc\,(\cdot)= e^{-\phi_1}\,(-\Delta)^2 e^{\phi_1}(\cdot) \quad \mbox{where $\phi_1$ fulfils \eqref{phi_conditions} for $m=2$.}
\end{align*}
This completes the proof of Theorem \ref{main_theorem_3} for $m=2$.\smallskip

\textbf{Step 2.} The case of $m=3$.\smallskip

The proof for  $ m=3$  is similar to the case of $m=2$. Therefore, in this case we do not give the complete details of the proof. By the assumption we have
 \begin{align}\label{main_identity_modified}
	        \int\limits_{\Omega} \sum_{l=0}^{5}a^{l}_{i_1\cdots i_l} D^{i_1\cdots i_l}u\, v=0\quad \mbox{where $u$ and $v$ solve} \quad  \Delta^3 u= \Delta^3 v=0.
	    \end{align}
Inserting the CGO form of $u$ and $v$  given in  \eqref{cgos2} and  multiplying \eqref{main_identity_modified} 
by $h^{5}$ and then letting $ h\rightarrow 0$ we obtain from \Cref{main_theorem_2} that
   $  a^5= \D^5\phi_0+ i_{\d} a^{5,1}$, where $ \phi_0$ fulfils $  \PD^l_{\nu}\phi_0|_{\PD\Omega}=0$ for $ 0\le l\le 4.$
 Next we multiply \eqref{main_identity_modified} by $ h^{4}$ and let $h\rightarrow 0$ and repeat similar analysis for the case of $m=2$ to  deduce
 $a^5= i_{\d} \D^3\phi_1 + i^2_{\d} a^{5,2},\,
    a^{4}=  \D^4 \phi_1+ i_{\d} a^{4,1}$ and $
    \phi_0=0$,
 where $\phi_1$ satisfies the boundary conditions $\PD^l_{\nu}\phi_1=0$ for $0\le l\le 3$. We now consider coefficients of $ h^{-3}$. To do that, we multiply \eqref{main_identity_modified} by $ h^{3}$ and let $h \rightarrow 0$ to obtain
\begin{align}\label{eq_5.5}
 0=& \int\limits_{\Omega} a_{i_1i_2i_3}^3(e_1+\I \eta)_{i_1}  \cdots (e_1+\I \eta)_{i_3} \, a_0 b_0+   \int\limits_{\Omega} a^{4,1}_{i_1i_2}(e_1+\I \eta)_{i_1}  (e_1+\I \eta)_{i_2} \, Ta_0 b_0+  \int\limits_{\Omega} a_{i_1}^{5,2}(e_1+\I \eta)_{i_1} \, T^2a_0 b_0\nonumber\\& \quad+4\int\limits_{\Omega} \D_{i_1\cdots i_4}^4 \phi_1\, (e_1+\I \eta)_{i_1}  \cdots (e_1+\I \eta)_{i_3} \, (\PD_{i_4}a_0) b_0 + 3\int\limits_{\Omega} \D_{i_1\cdots i_3}^3 \phi_1\, (e_1+\I \eta)_{i_1}   (e_1+\I \eta)_{i_2} \,(T (\PD_{i_3} a_0)) b_0 \nonumber \\& \quad\quad+\int\limits_{\Omega} \D_{i_1\cdots i_3}^3 \phi_1\, (e_1+\I \eta)_{i_1}  \cdots (e_1+\I \eta)_{i_3} \, \Delta a_0 b_0.
 \end{align}
 We next analyze the last three terms in the above expression. To this end, we make use of integration by parts and $ \PD_{\nu}^{l}\phi_1|_{\PD\Omega}=0$ (where $0\le l\le 3$), then followed by  $ T^3a_0=T^3b_0=0$. We are not going to mention where we applied these assumptions because it will be clear from the context. As a result, we obtain 
 \begin{equation*}
     \begin{aligned}
     \begin{cases}
     \int\limits_{\Omega} \D_{i_1\cdots i_3}^3 \phi_1\, (e_1+\I \eta)_{i_1}  \cdots (e_1+\I \eta)_{i_3} \, \Delta a_0 b_0
  = -\frac{1}{8}\int\limits_{\Omega} \phi_1\,\left(3\, T^2 \Delta a_0\, Tb_0+3\, T \Delta a_0\, T^2b_0  \right),\\
   \int\limits_{\Omega} \D_{i_1\cdots i_3}^3 \phi_1\, (e_1+\I \eta)_{i_1}   (e_1+\I \eta)_{i_2} \,T (\PD_{i_3} a_0)\, b_0
  =-\frac{1}{4}\int\limits_{\Omega}  \phi_1\,  \Big[2\, \left\langle  \nabla  T^2 a_0,\nabla T b_0\right\rangle+ \left\langle  \nabla T a_0, \nabla T^2 b_0\right\rangle \Big]\\\hspace{8cm} - \int\limits_{\Omega}  \phi_1\, \Big[2\,  T^2 \Delta a_0 \, Tb_0+ T\Delta a_0\,T^2b_0) \Big],\\
  \int\limits_{\Omega} \D_{i_1\cdots i_4}^4 \phi_1\, (e_1+\I \eta)_{i_1}  \cdots (e_1+\I \eta)_{i_3} \, (\PD_{i_4}a_0) b_0
  =\frac{1}{8}\int\limits_{\Omega} \phi_1 \left( 3\,T^2 \Delta a_0\, Tb_0+3\, T\Delta a_0 \, T^2 b_0\right)\\ \hspace{8cm}+\int\limits_{\Omega} \phi_1 \left( 3\,\langle \nabla T^2  a_0, \nabla Tb_0\rangle +3\,\langle \nabla T\Delta a_0,\nabla T^2 b_0\rangle \right).
  \end{cases}
     \end{aligned}
 \end{equation*}
Choosing $a_0=x\cdot \eta$, we see that $Ta_0=2\,\I$. This implies $P(D)Ta_0=0$, where $P(D) $ is any differential operator. This along with \eqref{eq_5.5} implies  \begin{align}\label{eq_8.4}
    0= & \int\limits_{\Omega} a_{i_1i_2i_3}^3(e_1+\I \eta)_{i_1}  \cdots (e_1+\I \eta)_{i_3} \, a_0 \,b_0+ \int\limits_{\Omega} a^{4,1}_{i_1i_2}(e_1+\I \eta)_{i_1}  (e_1+\I \eta)_{i_2} Ta_0\, b_0.
 \end{align}
 The combination of this  with \Cref{lm:summary_1}  implies
 \begin{align*}
 \begin{cases}
      a^3= \D^3 \psi_1+ i_{\d} a^{3,1} \quad\hspace{2mm} \mbox{with} \quad \PD^l_{\nu}\psi_1|_{\PD\Omega}=0  \quad \mbox{for}\quad  0\le l \le 2, \\
      a^{4,1}= \D^2 \psi_2 + i_{\d} a^{4,2} \quad \mbox{with} \quad \PD^l_{\nu}\psi_2|_{\PD\Omega}=0  \quad \mbox{for}\quad 0\le l \le 1.
      \end{cases}
 \end{align*}
To proceed further we next establish a relation between $ \psi_1$ and  $ \psi_2$.
To do so, we substitute above relations into    \eqref{eq_8.4} and  use integration  by parts to conclude
 \begin{align*}
     \frac{1}{8} \int\limits_{\Omega} \psi_1\, ( 3T^2a_0\, Tb_0+ 3 Ta_0\, T^2b_0) - \frac{1}{4}\int\limits_{\Omega} \psi_2\, ( 2T^2a_0\, Tb_0+  Ta_0\, T^2b_0)=0. 
 \end{align*}
Since we choose $a_0= x\cdot \eta$, this implies  $ T^2a_0=0$ but $ Ta_0\neq 0$. Next choosing $b_0$ such that $T^2b_0\neq 0$, we obtain
 $ 2 \psi_2-3\psi_1=0$ in $ \Omega$. Next we choose $a_0$ such that $T^2a_0= c$, for certain non-zero constant and $b_0$ such that $Tb_0=0$. With these choices of $a_0$ and $b_0$ from \eqref{eq_5.5} we obtain $\int_{\Omega} a_{i_1}^{5,2}(e_1+\I \eta)_{i_1} \, T^2a_0 b_0=0$. This implies $a^{5,2}= \D\psi_3$, where $ \psi_3=0$ on $\PD\Omega$. Thus we have 
 \begin{align*}
     \begin{cases}
       a^3= \D^3 \psi_1+ i_{\d} a^{3,1}, \quad   a^{4,1}= \frac{3}{2}\D^2 \psi_1 + i_{\d} a^{4,2} \quad \mbox{with} \quad \PD^l_{\nu}\psi_1|_{\PD\Omega}=0  \quad \mbox{for}\quad  0\le l \le 2\\
      a^{5,2}= \D\psi_3\quad  \mbox{with}\quad  \psi_3|_{\PD\Omega}=0.  
     \end{cases}
 \end{align*}
 Substituting this into \eqref{eq_5.5} and then choosing $a_0$ and $b_0$ such that $T^2a_0=c$ for some non-zero constant $c$ and $ T^2b_0 = 0$, we obtain
$\int_{\Omega} (4\,\psi_3+3\,\psi_1) ( T^2a_0\, Tb_0) =0.$
This entails $ \psi_3=-\frac{3}{4} 3\psi_1$. Thus we obtain \begin{align*}
     a^3= \D^3 \psi_1+ i_{\d} a^{3,1}, \quad
      a^{4,1}&= \frac{3}{2}\D^2 \psi_1 + i_{\d} a^{4,2}, \quad 
      a^{5,2}= -\frac{3}{4}\D\psi_1. \end{align*}
      
      Again inserting above findings into \eqref{eq_5.5} and  utilizing integration by parts and boundary conditions of $\psi_1$, we see that the first three integrals in \eqref{eq_5.5} disappear. As a result we end up with integral identities involving $\varphi_1$ and one can iterate similar arguments presented in \textbf{Step 2.} to conclude $\varphi_1=0$ in $\Omega$.
    This  entails $a^5=  -\frac{3}{4}i^2_{\d}\D\psi_1,
          a^4= i_{\d} (\frac{3}{2}\D^2 \psi_1 + i_{\d} a^{4,2})$ and $
          a^3=\D^3 \psi_1+ i_{\d} a^{3,1},$
      where $ \psi_1$ satisfies the boundary conditions $\PD^l_{\nu}\psi_1|_{\PD\Omega}=0$ for $ 0\le l \le 2. $ We next use the assumption from Theorem \ref{main_theorem_3} that $ \PD^l_{\nu} a^5|_{\PD\Omega}=0$ for $ 0\le l\le 5$ to conclude $\PD^l_{\nu}\psi_1|_{\PD\Omega}=0$ for $ 0\le l \le 5$. 
      In this stage we perform the gauge transformation and substitute $ a^5=0$. Then the new set of coefficients can be recovered uniquely by Theorem \ref{th:main_theorem_1}. In particular we have \begin{align*}
          \Lc \,(\cdot) = e^{-\psi_1}(-\Delta)^3 e^{\psi_1} (\cdot) \quad \mbox{where $\psi_1$ satisfies \eqref{phi_conditions} for $m=3$.}
      \end{align*}
      This completes  proof of Theorem \ref{main_theorem_3} for the case of $m=3$.\smallskip
    
 \textbf{Step 3.} The case of $m\ge 4$.\smallskip
 
In this case  $ a^{2m-1}= i_{\d}^{m-1}A^1$, for some vector field $A^1\in C^{\infty}(\overline{\Omega})$ by the assumption of Theorem \ref{main_theorem_3}.  This implies we have the following integral identity.
 \begin{align}\label{eq_6.12}
	    \int\limits_{\Omega}A_i^1 D^{i}(-\Delta)^{m-1}u\, v+    \int\limits_{\Omega} \sum_{l=0}^{2m-2}a^{l}_{i_1\cdots i_l} D^{i_1\cdots i_l}u\, v=0\quad \mbox{where} \quad  \Delta^mu= \Delta^m v=0.
	    \end{align} 

We next choose suitable CGO solutions  of the polyharmonic operators given in Lemma  \eqref{lm:cgo_forms}. Note that, Laplace operator $(-\Delta)$ acting on such solutions can produce maximum one negative power of $h$; see \eqref{CGO_1}. Thus $ D_i (-\Delta)^{m-1}$ will produce  maximum $m$-th negative power of $h$ when applied to CGO solutions. Therefore, we will encounter the vector field $A^1$ when we multiply \eqref{eq_6.12} by $h^m$ and the integral $ \int_{\Omega}A_i^1 D^{i}(-\Delta)^{m-1}u\, v$ will disappear when we multiply it by $h^{m+k}$ for some integer $k\ge 1$ and  let $h\rightarrow 0$.  We start by multiplying \eqref{eq_6.12} by $h^{2m-2},\cdots, h^{m+1}$  respectively  and utilize Theorem \ref{th:main_theorem_1} to derive
 $
     a^{2m-l}= i^{m-l}_{\d} b^{2m-l,m-l}$ for certain smooth symmetric tensors $b^{2m-l,m-l}$ of order $l$ defined in $\overline{\Omega}$ for each $l$ with $ 2\le l\le m-1$.
Inserting these into \eqref{eq_6.12} and then  multiplying \eqref{eq_6.12} by $h^m$ we obtain
 \begin{align}\label{eq_6.13}
    \int\limits_{\Omega} A_{i}^1(e_1+\I \eta)_i \, T^{m-1}a_0\, b_0+ \sum\limits_{l=2}^{m} \int\limits_{\Omega} b_{i_1\cdots i_l}^{2m-l,m-l}\, (e_1+\I \eta)_{i_1} \cdots (e_1+\I \eta)_{i_l}\, T^{m-l}a_0\, b_0=0.
 \end{align}

 We next write  above integrals over $\mathbb{R}^n$ by extending the symmetric tensor fields by $0$ outside $ \overline{\Omega}$. Then we choose particular solution of the transport equation $T^ma_0=T^mb_0=0$. 
\noindent Now combination of these, \eqref{eq_6.13}  and \Cref{lm:summary_1}  entail  
  $ A^1=  \D\phi_1 $ and $ b^{2m-l,m-l}= \D^l\phi_l+i_{\d} b^{2m-l,m-l+1}$ for $ 2\le l \le m$,
where each $\phi_l$ for $ 1\le l \le m$ satisfies the boundary conditions $\PD^{l-1}_{\nu} \phi_l|_{\PD\Omega}=0$.  
Substituting above relations into \eqref{eq_6.13}  we obtain $\int_{\Omega}
     \sum_{l=1}^m \left\langle \D^l\phi_l,(e_1+\I\eta)^{\otimes l}\right\rangle \,T^{m-l}a_0 \, b_0=0$.
This along with integration by parts and the boundary conditions for $\phi_l$ implies $ \sum_{l=1}^m (-1)^l\,\phi_l\, \sum_{p=1}^l \binom{l}{p} \, T^{m-p}a_0\,T^p b_0=0.$
We next choose $b_0$ in such a way that $T^{p+1}b_0=0$ while $ T^pb_0\neq 0$ for $1\le p \le m-1$. The combination of this with $\sum_{l=1}^m (-1)^l\,\phi_l\, \sum_{p=1}^l \binom{l}{p} \, T^{m-p}a_0\,T^p b_0=0$ implies
 %
 $\sum_{l=p}^{m} \binom{l}{p}\, (-1)^l\, \phi_l=0$ for each $p$ with $ 1\le p \le m-1 .$
 We obtain  $m-1$ relations involving  $m$ unknown. 
 Fix $x_0\in \Omega$, then from above we obtain
 \begin{align*}
     \sum\limits_{l=p}^{m} \binom{l}{p}\, (-1)^l\, \phi_l(x_0)=0 \quad \mbox{for} \quad 1\le p \le m-1 .
 \end{align*}
 It is evident that the above matrix equation has one dimensional kernel. Therefore  $\phi_1(x_0)=\phi_2(x_0)=\cdots =\phi_m(x_0)=\phi(x_0)\, \mbox{(say)}$. Since  the matrix is independent of $x$, this implies one can vary $x_0$ and obtain $\phi_1(x)=\phi_2(x)=\cdots =\phi_m(x)=\phi(x) $ in $\Omega$. Therefore we obtain $ a^{2m-1} =i^{m-1}_{\delta}A^1= i^{m-1}_{\delta}(\nabla \phi)$ with $ \PD^{m-1}_{\nu} \phi|_{\PD\Omega} =0$. By the assumption of Theorem \ref{main_theorem_3} we have $ \PD^{l}_{\nu} a^{2m-1}|_{\PD\Omega} =0$ for $0\le l\le 2m-1$, this immediately gives the same for $ \phi$. Since $\phi$ satisfies \eqref{phi_conditions}, this implies we can now implement the gauge transformation and eliminate the  highest order perturbation $a^{2m-1}$. Lower order coefficients will change as a result and those terms can be retrieved using Theorem \ref{th:main_theorem_1}. Thus we obtain 
  \begin{align*}
       \Lc (\cdot)= e^{-\phi}(-\Delta)^m e^{\phi}\, (\cdot) \quad \mbox{where $\phi$ fulfils} \quad \PD^{l}_{\nu} \phi|_{\PD\Omega} =0 \quad\mbox{for}\quad 0\le l\le 2m-1.
  \end{align*}
  This completes the proof.
 \end{proof}
\subsection{Remarks about general case.} We wrap up this section with a discussion regarding how one might proceed in the general case.  
Since the operator \eqref{operator} we consider is not symmetric, the gauge we encounter in the term $a^{2m-1}$ will propagate up to $a^0$. 
One expects gauges in the different orders:
\begin{align*}
    a^{2m-1}&\quad \mbox{ involves first order derivatives of some scalar function}\\
    a^{2m-2}& \quad\mbox{ involves second  order derivatives of some scalar function}\\
    \vdots\\
    a^{m} & \quad \mbox{ involves $m$th order derivatives of some scalar function}
\end{align*}

Here is one possible approach to recover the coefficients:
\begin{enumerate}
    \item  Multiplying $h^{2m-1}$ in the integral identity  \eqref{integral_identity_1} one can infer that
\begin{align*}
    a^{2m-1}= \D^{2m-1}\phi_0+i_{\d}a^{2m-1,1} \quad \mbox{where $ \phi_0$ satisfies $ \PD^{l}_{\nu}\phi_0=0$ for $0\le l \le 2m-2$}.
\end{align*}
\item Next multiplying $h^{2m-2}$ in \eqref{integral_identity_1} we obtain
\begin{align*}
    a^{2m-2}= \D^{2m-2}\phi_1+i_{\d} a^{2m-2,1},\quad a^{2m-1,1}= c_1\D^{2m-3}\phi_1+i_{\d}a^{2m-1,2}, \quad \phi_0=0 
\end{align*}
where $\phi_1$ satisfies $\PD^{l}_{\nu}\phi_l=0$ for $ 0\le l \le 2m-3$.
\item In general at $k$-th stage one can expect the following
\begin{align*}
   a^{2m-l} = c_{k,l}\,i^{m-1}_{\d} \D^{2m-2k+l} \phi_k +i^{k+1-l}_{\d} a^{2m-l,k+1-l} 
\end{align*}
holds for $ 1\le l\le k$ and $ 1\le k\le m$ and $ \phi_{l}=0$ for $ 1\le l \le k-1$.
\item  If the above assertion is proved, then substituting $ k=m,l=1$ in the above we obtain
\begin{align*}
    a^{2m-1}= i^{m-1}_{\d} \nabla \phi_m \quad  \mbox{with} \quad \PD^l_{\nu}\phi_m=0 \quad l=0,1.
\end{align*}
\item At this point, one can perform the gauge transformation to replace $ a^{2m-1}$ by $0$. The new set of coefficients up to order $ 2m-2$ can be then recovered by Theorem \ref{th:main_theorem_1}. 
\end{enumerate}

\section{Momentum ray transforms}\label{sec:mrt}
Let
$ S^m=S^m(C^{\infty}_{c}(\mathbb{R}^n)) $ be the space of smooth compactly supported symmetric $ m $-tensor fields in $ \R^n $. For $ f\in S^m$ and for each non-negative integer $k$, the momentum ray transform (MRT)  is denoted by $\mathcal{J}^k $ and given by
\newcommand{\J}{\mathcal{J}}
\begin{align*}
    \mathcal{J}^kf(x,\xi)\coloneqq \int_{-\infty}^{\infty} t^k \, f_{i_1\cdots i_m}(x+t\xi)\, \xi^{i_1}\,\cdots \xi^{i_m}\, \D t \quad \mbox{for all} \quad (x,\xi)\in \Rn\times (\Rn \setminus\{0\}).
    \end{align*}
    
    These transforms were first introduced by Sharafutdinov \cite{Sharafutdinov_1986_momentum,Sharafutdinov_book}. For $k=0$, $\J^0f=\J f$ is the classical ray transform/ X-ray transform, which is well studied due to its potential application in different branches of science such as medical imaging, seismology, and inverse problems related to partial differential equations. For $m=0$, MRT appears when studying the inversion of the cone transform and conical Radon  transform, where the latter transforms have promising applications in Compton cameras; see \cite{kuchment_cone_trans}. 
We also need to define the MRT of $ F\in \mathbf{S}^m=S^0\oplus  S^1\oplus\cdots\oplus S^m $. Any such $F$ can be written uniquely as
\begin{equation}\label{definition_of_F_m}
\begin{aligned}
	F &\coloneqq\sum_{p=0}^{m}f^{(p)}=f^{(0)}_{i_0}+f^{(1)}_{i_1}\,\D x^{i_1} + \cdots+f^{(p)}_{i_1\cdots i_p}\D x^{i_1}\cdots \D x^{i_p}+\cdots+f^{(m)}_{i_1\cdots i_m}\D x^{i_1}\cdots \D x^{i_m}\\
    & \hspace{1.5cm}	\quad=\left( f^{(0)}_{i_0},f^{(1)}_{i_1},\cdots, f^{(m)}_{i_1\cdots i_m} \right)
\end{aligned}
\end{equation} which can be viewed as sum of a function, a vector field, and symmetric tensor fields with $ f^{(p)}\in S^p$ for each $0\leq p\leq m$. For any smooth and compactly supported $ F\in \mathbf{S}^m$ and for  all $ (x,\xi) \in \R^n \times \R^n\setminus \{0\}$ and for every integer $ k\ge0$  the momentum ray transform \cite{polyharmonic_mrt_application} is given by 
\begin{equation}
	\begin{aligned}\label{Eq4.1}
		J^{k}F(x,\xi) &\coloneqq\sum_{p=0}^{m}\,J^k\!f^{(p)}(x,\xi)\\
		&\coloneqq \int_{-\infty}^{\infty} t^k \left(f^{(0)}_{i_0}(x+t\xi
		)+ f^{(1)}_{i_1}(x+t\xi)\, \xi^{i_1}+\cdots+ f^{(m)}_{i_1\cdots i_m}(x+t\xi)\, \xi^{i_1}\,\cdots \xi^{i_m}\right)\D t.
	\end{aligned}
\end{equation}

\newcommand{\sm}{\mathbf{S}^m}
We now introduce the notion of MRT on the set $\sm(\mathcal{E}') $ consisting of compactly supported distributions in $\sm$. This will be done by using adjoints. The momentum ray transforms $ J^{k}:\mathbf{S}^m(\mathcal{E}') \rightarrow \mathcal{D}'(\Rn \times \Rn\setminus\{0\}) $ \cite[Definition 4.4]{polyharmonic_mrt_application} are defined as:
	\begin{equation*}
		\begin{aligned}
			\langle J^{k}F,\Psi\rangle \NT = \NT \langle F,(J^{k})^*\Psi\rangle = \sum\limits_{p=0}^{m}\langle f^{(p)}, ( J^{k} )^{*}_{p} \Psi\rangle \quad \mbox{where}\quad \Psi \in C_c^{\infty}(\Rn\times \Rn\setminus\{0\}), 
		\end{aligned}
	\end{equation*}%
and  $ (J^{k})^*\Psi  $ is given by
	\begin{align*}
		\notag ((J^{k})^*\Psi)(x)= \Bigg{(} &\int_{\Rn}\int_{\mathbb{R}}t^k\, \, \Psi(x-t\xi,\xi) \D t\,\D \xi,\cdots,
		\int_{\Rn}\int_{\mathbb{R}}t^k\, \, \xi^{i_1}\cdots \xi^{i_m}\, \Psi(x-t\xi,\xi) \D t\,\D \xi\Bigg{)},
	\end{align*}%
    	which is an element of  $\oplus_{p=0}^{m} C^{\infty}(S^p)$ and 
    	\[
    	(J^{k})^{*}_{p} \Psi= \int_{\Rn}\int_{\mathbb{R}}t^k\, \, \xi^{i_1}\cdots \xi^{i_p}\, \Psi(x-t\xi,\xi) \D t\,\D \xi.
    	\]

\subsection{Injectivity of MRT}
We denote $I^{k}F\coloneqq J^{k}F|_{\Rn\times \Sb^{n-1}}$ for every integers $k\ge 0$.
In this section we recall certain injectivity results of the operator $ F \longmapsto J^{m}F$ and $F\longmapsto I^{m}F $ from \cite{polyharmonic_mrt_application} without their proofs. 
\begin{lemma}\cite[Theorem 4.5]{polyharmonic_mrt_application}\label{uniquenes_result_mrt}
		Let $ F \in \mathbf{S}^m(\mathcal{E}')$ and $ m\ge 0$. Then  
		\begin{align*}
		     J^{k}F=0  \quad \mbox{holds for all integers $k$ with $0\le k \le m$} \quad \implies\quad F=0.
		\end{align*}
	\end{lemma}
	\begin{lemma}\cite[Lemma 4.8]{polyharmonic_mrt_application}\label{translation_lemma}
Let $F \in \mathbf{S}^m(\mathcal{E}')$, then the operator $ J^{k}F$ satisfies the  following relation
	\begin{align*}
		\langle\xi,\frac{\PD}{\PD x} \rangle^{p}  J^{k}F=  \begin{cases}
			(-1)^p\,\binom{k}{p}\,p!\, J^{k-p}F \quad \mbox{if} \quad p\le k\\
			0 \hspace{3.3cm} \mbox{if} \quad p>k.
		\end{cases} 
	\end{align*}
	After taking restrictions on $ \Rn\times \mathbb{S}^{n-1}$ this gives
	\begin{align*}
	    	\langle\xi,\frac{\PD}{\PD x} \rangle^{p}  I^{k}F=  \begin{cases}
			(-1)^p\,\binom{k}{p}\,p!\, I^{k-p}F \quad \mbox{if} \quad p\le k\\
			0 \hspace{3.3cm} \mbox{if} \quad p>k.
		\end{cases} 
	\end{align*}
\end{lemma}
This lemma entails that highest order MRT $(J^mF)$ uniquely determines all the lower order MRT $(J^kF)\, 0\le k\le m$. Analogous result holds if one replace $J^kF$ by $I^kF$. Consequently, one can obtain next results.
\begin{lemma}\label{unique_based_on_highest_mrt}
    	Suppose $ F \in \mathbf{S}^m(\mathcal{E}')$ and $m\ge 0$. Then 
	$J^{m}F=0 \implies  F=0.$
\end{lemma}

\begin{lemma}\cite[Theorem 4.18]{polyharmonic_mrt_application}\label{mrt_uniqueness_2}
    Let $m\ge 2$ and $F =\sum\limits_{l=0}^{m} f^{(l)} \in \mathbf{S}^m(\mathcal{E}')$. 
Then  
\[I^{m}F= J^{m}F|_{\Rn\times \Sb^{n-1}}=0\]
if and only if 
\begin{align*}
      f^{\lr{2\lb{\frac{m}{2}}}}
=- i_{\d}\left( \sum\limits_{l=1}^{\lb{\frac{m}{2}}} i^{\lb{\frac{m}{2}}-l}_{\d} f^{(2l-2)}\right) \quad \mbox{and} \quad
    f^{\lr{2\lb{\frac{m-1}{2}}+1}}&= -i_{\d}\left( \sum\limits_{l=1}^{\lb{\frac{m-1}{2}}} i^{\lb{\frac{m}{2}}-l}_{\d} f^{(2l-1)}\right).
\end{align*}
Moreover, $I^{m}F=0$ if and only if $F=0$ for $m=0,1$.
\end{lemma}
\begin{lemma}\label{increasing_order}
The following equality 
\begin{align*}
       I^k (i^p_{\d}f^{(l)})(x,\xi) = I^k f^{(l)}(x,\xi)
   \end{align*}
   holds for any $f^{(l)} \in S^l $ and for all integers $p$ with $ p\ge 1$. For $p=0$ this holds trivially because $i^0_{\delta}$ is the identity operator.
\end{lemma}


\begin{remark}
   We know that any symmetric $m$-tensor field in $\Rn$ has $\binom{m+n-1}{m}$ distinct components. Suppose $ F$ is the same as in \eqref{definition_of_F_m}. Then $F$ has a total of $ \sum\limits_{l=0}^m \binom{l+n-1}{l}= \binom{m+n}{m}$ distinct components. As a consequence, recovering $F$ in $\mathbb{R}^n$ is  equivalent to  recovering  a symmetric   $m$-tensor field in  $\mathbb{R}^{n+1}$. However to recover a symmetric $m$-tensor field in $\mathbb{R}^{n+1}$, MRT information is  required on the tangent bundle of unit sphere which is a $2(n+1)-2=2n$ dimensional data set. 
   As a result, the kernel of $J^mF$ is trivial (see \Cref{unique_based_on_highest_mrt}), whereas the kernel of $I^mF$ is nontrivial when  $m\ge 2$ (see \Cref{mrt_uniqueness_2}) because $I^mF$ is specified  on the unit sphere bundle.  
   
 
\end{remark}

Next we prove the following result, which was used to segregate MRT of different order tensor fields in the previous sections.
\begin{lemma}\label{lem:sum_of_mrt}
Let $F_m\in \mathbf{S}^m(\mathcal{E}')$ and  $m\ge 0$. 
\begin{itemize}
    \item [1.] Suppose \begin{align*}
    c_mI^{k+m}F_m+\cdots+ c_1 I^{k+1}F_1+ c_0 I^{k}F_0 =0 
\end{align*}
holds for certain non-zero  constants $c_j$ where $ 0\le j\le m$ and for all $ 0\le k\le m$. Then 
\begin{align}\label{each_mrt_relation}
    I^pF_p=0 \quad \mbox{for} \quad p=0,1,\cdots,m.
\end{align}
\item [2.] Additionally, if we assume
 \begin{align*}
      d_m I^{k+m-1}F_m+\cdots+ d_1 I^{k}F_1=0 
 \end{align*}%
  holds for certain non-zero constants $d_j$ where $ 0\le j\le m-1$ and for all $ 0\le k\le m-1$. Then
 \begin{align}\label{eq_7.10}
      I^pF_{p+1}=0 \quad \mbox{for} \quad p=0,1,\cdots,m-1.
 \end{align}
\end{itemize}
\end{lemma}
\begin{proof}
We only give the  proof of \eqref{each_mrt_relation} and that of \eqref{eq_7.10} follows similarly. 
We have that 
\begin{align*}
     c_mI^{k+m}F_m+\cdots+ c_1 I^{k+1}F_1+ c_0 I^{k}F_0 =0.
\end{align*}
    Applying $ \langle \xi,\PD_x\rangle^k\,  $ to the above equation and then using Lemma \ref{translation_lemma} we obtain
    \begin{align*}
    &    c_m \binom{k+m}{k}\,I^{m}F_m+c_{m-1}\, \binom{k+m-1}{k}\,I^{m-1}F_{m-1}+\cdots+c_1\,\binom{k+1}{k}\, I^{1}F_1+ c_0\,\binom{k}{k}\, I^{0}F_0=0,
    \end{align*}
    for all $k$ with $0\le k\le m$.
    The above relation can be written as the matrix equation $AX = 0$ where 
    \begin{equation*}
        \begin{aligned}
          X&= (I^mF_m,\cdots,I^0F_0)^t,\\
            A&=\begin{pmatrix}
            c_m & c_{m-1}& \cdots& c_1& c_0\\
           c_m \binom{m+1}{1} & c_{m-1}\binom{m}{1}& \cdots& c_1\binom{2}{1}&c_0  \smallskip\\
          c_m \binom{m+2}{2} & c_{m-1}\binom{m+1}{2}& \cdots&c_1\binom{3}{2}& c_0\smallskip\\
           \vdots&\vdots&\cdots&\vdots&\vdots\\
           c_m \binom{2m}{m} & c_{m-1}\binom{2m-1}{m}& \cdots&c_1\binom{m+1}{m}& c_0
            \end{pmatrix}.\end{aligned}
    \end{equation*}
    We now complete the proof by showing  $\det A\neq 0$. This follows from Lemma \ref{prop:determinant}. This implies $X=0$, i.e. 
    \[I^p F_p=0 \quad \mbox{for} \quad p=0,1,\cdots,m.\]
\end{proof}

\begin{lemma}\label{prop:determinant}
    $\det A = (-1)^{\frac{m(m+1)}{2}}\, c_0\cdots c_m  $.
\end{lemma}
\begin{proof}
We have 
    \begin{align*}
        \det A = c_0\cdots c_m  \det A_{m+1}
    \end{align*}
    where \begin{align*}
       A_{m+1}&=\begin{pmatrix}
            1 & 1& \cdots&1& 1\\
           \binom{m+1}{1} & \binom{m}{1}& \cdots& \binom{2}{1}&1 \smallskip \\
          \binom{m+2}{2} & \binom{m+1}{2}& \cdots& \binom{3}{2}\, &1\\
           \vdots&\vdots&\cdots&\vdots\\
           \binom{2m}{m} & \binom{2m-1}{m}& \cdots&  \binom{m+1}{m}&1 
            \end{pmatrix}.
    \end{align*}
    We now subtract the $ (m-k)$-th column from the $(m-k-1)$-th column for $1\le k \le m-2 $ to obtain
    \begin{align*}
        A_{m+1}\longrightarrow \begin{pmatrix}
            1 & 0& \cdots& 0\\
           \binom{m+1}{1} & -1& \cdots& -1 \smallskip \\
          \binom{m+2}{2} & -\binom{m+1}{1}& \cdots& -\binom{2}{1}\, \\
           \vdots&\vdots&\cdots&\vdots\\
           \binom{2m}{m} & -\binom{2m-1}{m-1}& \cdots&  -\binom{m}{m-1} 
            \end{pmatrix}.
    \end{align*}
    This implies 
    $
        \det A_{m+1}= (-1)^{m} \det A_{m},
  $
    where $A_{m}$ is obtained from $A_{m+1}$ by considering first $m$ rows and $m$ columns respectively. Proceeding in this way after finally many steps we obtain
    $
        \det A_{m+1} = (-1)^{m+(m-1)+\cdots+2+1}. 
  $
    This implies 
   $
        \det A= (-1)^{\frac{m(m+1)}{2}}\, c_0\cdots c_m.
    $
    \end{proof}

\section{A new decomposition of symmetric tensor fields}\label{sec:decomposition}

We now prove a suitable trace free Helmholtz type decomposition of symmetric tensor fields, which we used in the previous section to deal with partial data MRT.
\begin{lemma}\label{trace_free_helmholtz decomposition}
    Let $f$ be a   smooth symmetric $m$ tensor field in $\overline{\Omega}$. Then the following decomposition holds:
    \begin{align*}
        f= \tilde{f}+i_{\delta} v+\D^{m}\phi,\quad \mbox{with}\quad j_{\delta}\tilde{f}=\delta^m\tilde{f}=0 \end{align*}
    and $\PD_{\nu}^{l}\phi|_{\PD \Omega}=0, l=0,1,\cdots,m-1$, where $\tilde{f} \in S^m$, $v \in S^{m-2}$ and $\phi \in S^0$ are smooth symmetric tensor fields in $\overline{\Omega}$.
    \end{lemma}
    This is a generalization of the decomposition shown in \cite{Dairbekov_Sharafutdinov}. In \cite{Rohit_Suman}, they proved  certain  decomposition  of symmetric tensors.  
\begin{remark}
    For $m=1$, this is the well known Helmholtz (or solenoidal) decomposition. 
\end{remark}
\begin{proof}
    We closely follow the arguments used in \cite{Dairbekov_Sharafutdinov}. We first assume that $f$ can be written in the given form as
    \begin{align}\label{assumtion_on_decom}
       f= \tilde{f}+i_{\delta} v+\D^{m}\phi,
    \end{align}
    and we derive an equation for $\phi$. Applying $j_{\delta}$ to \eqref{assumtion_on_decom} and using $j_{\delta} \tilde{f} = 0$ we get
    \[
         j_{\delta} f = \jd \id v + \jd \D^m \phi.
    \]
    Since $j_{\delta}i_{\delta}$ is invertible by \cite[Lemma 2.3]{Dairbekov_Sharafutdinov}, we obtain 
    \begin{align}\label{expre_of_v}
         v = (\jd \id )^{-1} ( \jd f-\jd\D^m \phi).
    \end{align}
    This together with \eqref{assumtion_on_decom} implies
    \begin{align}\label{eq_6.10}
        f&= \tilde{f} + \id ((\jd \id )^{-1} ( \jd f-\jd\D^m \phi)) +\D^m\phi.
    \end{align}
    Denote $ q = \id (\jd\id)^{-1} \jd $ and $p= (Id-q)$. From \cite[equation (2.15)]{Dairbekov_Sharafutdinov} we have that $p$ is the projection to the trace free component of a symmetric tensor and that $p(iv) = 0$ for any $v$. Thus from \eqref{eq_6.10} we obtain
    \begin{align*}
        pf&= \tilde{f}+p \D^m \phi.
    \end{align*}
    Applying $\delta^m$ to the above identity and using $\delta^m \tilde{f} = 0$ we obtain
    \begin{align*}
        \delta^m p\D^m \phi = \delta^m pf.
    \end{align*}
Hence, $\phi$ solves the following boundary value problem.              
\begin{equation}\label{bdry_value_problem}
    \begin{aligned}
  (-1)^m \delta^m p\D^m \phi &= (-1)^m\delta^m pf \quad \mbox{in } \quad \Omega\\
   \PD^l_{\nu}\phi&=0 \quad \hspace{1.8cm}\mbox{on} \quad \PD\Omega\quad \mbox{for} \quad l=0,1,\cdots,m-1.
\end{aligned}
\end{equation}

We next show that  $ \delta^m p\D^m$ acting on scalar functions in $\Omega$ is a strongly elliptic operator in the sense of \cite[Formula 11.79, Section 5.11]{Taylor_book}. Then by \cite[Exercise 3, Section 5.11]{Taylor_book}  the map
\begin{equation}\label{mapping}
    \begin{aligned}
        H^{2m}(\Omega)\cap H_0^{m}(\Omega)&\longrightarrow L^2(\Omega) \\
        \phi &\longrightarrow \delta^m p\D^m \phi
    \end{aligned}
\end{equation}
will be Fredholm operator of index zero. To prove the ellipticity of $\delta^m p\D^m $ we follow \cite[Section 5]{Dairbekov_Sharafutdinov} and introduce the following.
For $x \in \Rb^n$,  we define the symmetric multiplication operators $i_x :S^m  \rightarrow S^{m+1}$ by 
$$(i_x f)_{i_1i_2\dots i_{m+1}} =\sigma(i_1, \dots, i_m,  i_{m+1})(x_{i_{m+1}}f_{i_{1}i_{2}\dots i_{m}}).$$
We also define the dual of the operator $i_x$, the contraction operator $j_x :S^m  \rightarrow S^{m-1}$ by 
$$(j_x f)_{i_1i_2\dots i_{m-1}} =f_{i_1i_2\dots i_m}x^{i_m}.$$ Similarly, $i_{x^{\otimes k}}$ and $j_{x^{\otimes k}}$ are defined by taking the composition of $ i_{x}$ and $j_{x}$ with itself $k$ times, respectively. 
The principal symbol of $\delta^m p\D^m$ at $ (x,\xi)\in \Omega \times (\mathbb{R}^n\setminus\{0\})$ is given by $ (-1)^m\,j_{\xi^{\otimes m}} p i_{\xi^{\otimes m}}$, where $(\I)^m\,j_{\xi^{\otimes m}}$ and $ (\I)^m\,i_{\xi^{\otimes m}}$ are the principal symbols of $ \delta^m$ and $\D^m$ respectively. Thus the principal symbol of $ (-1)^m \delta^m p\D^m $ is $j_{\xi^{\otimes m}} p i_{\xi^{\otimes m}} $, which is real valued and non-negative since for any $\phi$ 
\begin{align*}
    \langle j_{\xi^{\otimes m}} p i_{\xi^{\otimes m}} \phi, \phi \rangle = \langle  p i_{\xi^{\otimes m}} \phi,i_{\xi^{\otimes m}} \phi \rangle = \langle  p i_{\xi^{\otimes m}} \phi,pi_{\xi^{\otimes m}} \phi \rangle.
\end{align*}
Also by Lemma \ref{lm:injectivity_symbol} we obtain $ j_{\xi^{\otimes m}} p i_{\xi^{\otimes m}} \neq 0 $ for $\xi \neq 0$. Thus we have shown the ellipticity of the boundary value problem \eqref{bdry_value_problem}. By elliptic regularity any element in the kernel of \eqref{bdry_value_problem} will be smooth. Now
$(-1)^m \delta^m p\D^m \phi =0$ in $\Omega$ and $ \PD^l_{\nu} \phi|_{\PD\Omega} =0$ for  $0\le l\le m-1$ gives
\begin{align*}
    \langle (-1)^m \delta^m p\D^m \phi,\phi \rangle =     \langle  p\D^m \phi,p \D^m \phi \rangle = 0 \quad \mbox{using integration by parts}.
\end{align*}
This implies 
$ pd^m \phi =0$. From  \cite[Lemma 3.4, Equation 3.8]{Dairbekov_Sharafutdinov} we have that
\begin{align*}
     \D p f = p\D f + \frac{m}{n+2m-2} i\delta p f \quad \mbox{for} \quad f\in S^m.
\end{align*}
Replacing $f$ by $\D^{m-1}\phi$ entails
\begin{align*}
     \D p \D^{m-1}\phi=  \frac{m-1}{n+2m-4} i\delta p \D^{m-1} \phi. 
\end{align*}
Denote $ u =p \D^{m-1}\phi$ and $v = \frac{m-1}{n+2m-4} \delta p \D^{m-1} \phi $. Note that, trace of $u=0$. This gives
\begin{align*}
    \D u= i v.
\end{align*}
Hence $ u$ is a { \it trace free conformal Killing tensor field}. Now using the boundary conditions $ \PD^l_{\nu}\phi|_{\PD\Omega} =0$ for  $0\le l\le m-1$ we conclude $ u|_{\partial \Omega}=0$. By \cite[Theorem 1.3]{Dairbekov_Sharafutdinov} we have $ u=0$. Repeating this process finitely many times we will obtain $ p\D\phi= \D\phi=0$. Since $\phi|_{\PD\Omega}=0$, this implies that $\phi=0$ in $\Omega$. Thus the boundary value problem has \eqref{bdry_value_problem} has zero kernel. Since it has Fredholm index zero, this says \eqref{bdry_value_problem}  also has zero co-kernel. This shows that the mapping \eqref{mapping} is an isomorphism.

Finally, given $f \in S^m$, the fact that \eqref{mapping} is an isomorphism together with ellipticity shows that there exists a unique $\phi \in C^{\infty}(\overline{\Omega})$ solving \eqref{bdry_value_problem}. We then define $v$ by \eqref{expre_of_v}, which implies that $i_{\delta} v = qf - q d^m \phi$. We also define $\tilde{f} = f - i_{\delta}v - d^m \phi$. It follows that 
\[
\tilde{f} = f - (qf - q d^m \phi) - d^m \phi = pf - p d^m \phi.
\]
Then $j_{\delta} \tilde{f} = 0$ and $\delta^m \tilde{f} = \delta^m pf - \delta^m p d^m \phi = 0$. This proves the required decomposition.
\end{proof}

\begin{lemma}\label{lm:injectivity_symbol}
    For any $\xi \neq 0$, if \begin{align*}
    p i_{\xi^{\otimes m}} \phi =0,
\end{align*}
then $\phi =0$.
\end{lemma}
\begin{proof}
 For $m=1$ we have $ pi_{\xi} \phi =i_{\xi}\phi$. Since $\xi \neq 0$ and $ (i_{\xi} \phi)_k = \xi_k \phi$, we obtain $\phi=0$. We proceed by induction and assume that  \begin{align*}
      \xi \neq 0, \ p i_{\xi^{\otimes m}} \phi =0 \implies \phi = 0.
 \end{align*}%
 Suppose that $p i_{\xi^{\otimes m+1}} \phi =0$. Then by \cite[Lemma  5.3]{Dairbekov_Sharafutdinov}we have
 \begin{align*}
     i_{\xi}p i_{\xi^{\otimes m}} \phi-\frac{2}{m+1}i_{\d} (j_{\d}i_{\d})^{-1} j_{\xi} pi_{\xi^{\otimes m}} \phi&=0. 
 \end{align*}
 Denote $f=p i_{\xi^{\otimes m}} \phi$. Taking the inner product of the above equation with $i_{\xi} f$ gives 
 \begin{align}\label{eq_6.16}
     &\langle i_{\xi} f, i_{\xi} f\rangle - \frac{2}{m+1} \langle i_{\d} (j_{\d}i_{\d})^{-1} j_{\xi} f, i_{\xi}f \rangle=0\nonumber\\
     \implies \quad &\langle j_{\xi}i_{\xi} f, f\rangle - \frac{2}{m+1} \langle (j_{\d}i_{\d})^{-1} j_{\xi} f, j_{\d}i_{\xi}f \rangle=0.
 \end{align}
 From \cite[Lemma 3.3.3]{Sharafutdinov_book} we have
 \begin{align*}
     j_{\xi}i_{\xi} f = \frac{|\xi|^2}{m+1}f+\frac{m}{m+1} i_{\xi} j_{\xi} f.
 \end{align*}
 Combining $j_{\d} f=0$ with the formula $j_{\d} i_{\xi} =\frac{2}{m+1} j_{\xi} + \frac{m-1}{m+1} i_{\xi}j_{\d} $ on $S^m$ given in \cite[Equation 5.6]{Dairbekov_Sharafutdinov}, we get
 \begin{align*}
     j_{\d} i_{\xi}f =\frac{2}{m+1} j_{\xi}f.
 \end{align*}
 From \cite[Equation 5.13]{Dairbekov_Sharafutdinov}, which can be used since $j_{\delta} f = 0$, we obtain
 \begin{align*}
     (j_{\d} i_{\d})^{-1} j_{\xi} f = \frac{m(m+1)}{2(n+2m-2)}j_{\xi}f.
 \end{align*}
 The combination of last two displayed equations and \eqref{eq_6.16} give
 \begin{align*}
     |\xi|^2 |f|^2 + m \left( 1- \frac{2}{n+2m-2}\right) |j_{\xi}f|^2=0.
 \end{align*}
 Since $ \big(1- \frac{2}{n+2m-2} \big) \ge 0 $ for  $n\ge 2$ and $m\ge 1$, this implies $ f= p i_{\xi^{\otimes m}} \phi=0$. Hence by induction we have $ \phi=0$. 
\end{proof}
\appendix

\section{construction of special solutions}\label{sec:preliminary}
In this section, we describe the construction of special solutions known as \emph{complex geometric optic (CGO)} solutions  of polyharmonic operators that we have already utilized in the earlier sections.

\subsection{Carleman estimate}
Let $\Omega\subset \R^n$, $n\geq 3$ be a bounded domain with smooth boundary. In this section, we construct Complex Geometric Optics (CGO) solutions for \eqref{operator} following a Carleman estimate approach from \cite{BUK,DOS}. 

We first introduce semiclassical Sobolev spaces. 
For $h>0$ be a small parameter, we define the semiclassical Sobolev space $H^s_{\mathrm{scl}}(\mathbb{R}^n)$, $s \in \R$ as the space $H^s(\Rn)$ endowed with the semiclassical norm
\begin{align*}
  \lVert u \rVert^2_{H^s_{\mathrm{scl}}(\Rn)} = \lVert \langle hD\rangle^s \, u \rVert^2_{L^2(\Rn)}, \quad \langle \xi \rangle= (1+|\xi|^2)^{\frac{1}{2}}.  
\end{align*}
For open sets $\Omega \subset\Rn$ and for non-negative integers $m$, the semiclassical Sobolev space $H^m_{\mathrm{scl}}(\Omega)$ is the space $H^m(\Omega)$ endowed with the following semiclassical norm
\begin{align*}
    \lVert u \rVert^2_{H^m_{\mathrm{scl}}(\Omega)} = \sum\limits_{|\alpha|\le m} \lVert ( hD)^{\alpha}\, u \rVert^2_{L^2(\Omega)}.
\end{align*}
These two norms are equivalent  when $\O=\Rb^n$ and for every integers $m\ge0$.


Let $\wt{\O}$ be an open subset containing $\O$ in its interior and let $\vp\in C^{\infty}(\wt{\O})$ with $\n \vp\neq 0$ in $\overline{\O}$. We consider the semi-classical conjugated Laplacian $P_{0,\vp} = e^{\frac{\vp}{h}} (-h^2 \Delta) e^{\frac{-\vp}{h}}$, where $0<h\ll 1$. The semi-classical principal symbol of this operator $P_{0,\vp}$ is given by $p_{0,\vp}(x,\xi)=|\xi|^2 - |\nabla_x\vp|^2 + 2\I \xi\cdot\nabla_x\varphi$.
\begin{definition}[\cite{KEN}] \label{LCW}
	We say that $\vp\in C^{\infty}(\wt{\O})$ is a limiting Carleman weight for $P_{0,\vp}$
	in $\O$ if
	$\nabla\vp \neq 0$ in $\overline{\O}$ and $\mathrm{\mathrm{Re}}(p_{0,\vp})$, $\mathrm{Im}(p_{0,\vp})$ satisfies
	\[\Big{\{}\mathrm{Re}(p_{0,\vp}), \mathrm{Im}(p_{0,\vp})\Big{\}}(x,\xi)=0 \mbox{ whenever } p_{0,\vp}(x,\xi)=0 \mbox{ for } (x,\xi)\in \widetilde{\O}\times(\mathbb{R}^n\setminus\{0\}),
	\]
	where $\{\cdot,\cdot\}$ denotes the Poisson bracket.
\end{definition}
Examples of such $\vp$ are linear weights
$\vp(x) = \A\cdot x$, where $0\neq \A\in\mathbb{R}^n$ or logarithmic weights $\vp(x)= \log|x-x_0|$ with $x_0 \notin \overline{\widetilde{\O}}$.

As mentioned already, we consider the limiting Carleman weight in this paper to be $\vp(x)= x_1$.
\begin{proposition}[Interior Carleman estimate]\label{Prop: Interior Carleman Estimate}
Let $\vp(x)$ be a limiting Carleman weight for the conjugated semiclassical Laplacian. Then there exists a constant $C=C_{\O, A^j, q}$ such that for $0< h\ll 1 $, we have
	\begin{equation*}
		h^{m}\lVert u\rVert_{L^2(\Omega)} \leq C \lVert h^{2m}e^{\frac{\vp}{h}}(-\Delta)^m e^{-\frac{\vp}{h}}u\rVert_{H^{-2m}_{\mathrm{scl}}},\quad \mbox{ for all } u\in C^{\infty}_0(\Omega).
	\end{equation*}%
\end{proposition}
This follows by iterating a Carleman estimate for the semiclassical Laplacian with a gain of two derivatives proved in \cite{Salo-Tzou}. We omit the proof here.  We refer the reader to \cite{KRU1,KRU2,Ghosh-Krishnan}.
\subsection{Construction of CGO solutions}\label{CGO_construction}
Next we use Proposition \ref{Prop: Interior Carleman Estimate} to construct CGO solutions for the equation $(-\Delta)^m u=0$. To this end, we  state an existence result whose proof  is standard; see \cite{DOS,KRU2} for instance.
\begin{proposition}\label{Prop_Existence}
Let  $\vp$ be as defined in Proposition \ref{Prop: Interior Carleman Estimate}. Then for any $v \in L^{2}(\Omega)$ and small enough $h>0$ one has $u \in H^{2m}(\Omega)$ such that
\begin{equation*}
	e^{-\frac{\vp}{h}}(-\Delta)^m e^{\frac{\vp}{h}} u = v \quad \mbox{in }\Omega, \quad \mbox{with}\quad \|u\|_{H^{2m}_{\mathrm{scl}}(\Omega)} \leq Ch^{m}\|v\|_{L^2(\Omega)}.
\end{equation*}%
\end{proposition}

We now use Proposition \ref{Prop_Existence} to construct a solution for $(-\Delta)^m\,u=0$ of the form
\begin{align}\label{CGO_0}
	u(x;h) &= e^{\frac{\vp + \I \psi}{h}}\left(a_0(x) + ha_1(x) + \dots + h^{m-1}a_{m-1}(x) + r(x;h)\right),\\
	\notag &=e^{\frac{\vp + \I \psi}{h}}( A(x;h) + r(x;h)), \quad\mbox{where } A(x;h)= \sum_{j=0}^{m-1} h^ja_j(x).
\end{align}
Here $\{a_j(x)\}$ and $r(x;h)$ will be determined later. 
We choose $\psi(x) \in C^{\infty}(\overline{\Omega})$ in such a way that
$p_{0,\vp}(x,\nabla\psi) =0$. This implies $
|\nabla \vp| = |\nabla \psi|$  and $\nabla\vp\cdot\nabla\psi = 0$ in $\Omega$. We calculate the term $(-\Delta)^m e^{\frac{\vp+i\psi}{h}}A(x;h)$ in $\Omega$.
Due to the choices of $\vp$ and $\psi$ we have $\nabla(\vp+\I\psi) \cdot \nabla(\vp+\I\psi) = 0$ in $\Omega$ and thus we obtain
\begin{equation}\label{CGO_1}
\begin{aligned}
e^{-\frac{\vp+i\psi}{h}}(-\Delta)^m e^{\frac{\vp+i\psi}{h}}A(x;h)
=& \left(-\frac{1}{h}T - \Delta \right)^{m} A(x;h) 
\end{aligned}
\end{equation}
where
\begin{equation}\label{term_H}
\begin{gathered}
	T = 2\nabla_x(\vp+\I\psi)\cdot\nabla_x + \Delta_x(\vp+\I\psi),
\end{gathered}
\end{equation}
We now make the coefficient of $h^{-m+j}$ to be $0$ for $0\le j \le m-1$ and obtain the following system of transport equations.
\begin{align}\label{CGO_2.0}
\begin{cases}
    T^m a_0(x) = 0 \quad \hspace{3cm} \mbox{in }\Omega,	\\
T^m a_j(x) = -\sum_{k=1}^{j} M_{k} a_{j-k}(x) \quad \mbox{in }\Omega, \quad \mbox{and}\quad  1\leq j\leq m-1.	
\end{cases}
\end{align}
where $M_j$'s are certain differential operator of order $m+j\, (1\le j\le m-1)$ and can be computed  from \eqref{CGO_1}.


It is well known that equations in \eqref{CGO_2.0}  have smooth solutions; see for instance \cite{DOS}.  We provide an explicit form for the smooth solution $a_0$ solving \eqref{CGO_2.0} for our inverse problem, which will be effective in getting the generalized MRT of the coefficients.
Using $a_0(x),\cdots,a_{m-1}(x) \in C^{\infty}(\Omega)$ satisfying \eqref{CGO_2.0}, we see that
\[  e^{-\frac{\vp+\I\psi}{h}}(-\Delta)^m e^{\frac{\vp+\I\psi}{h}}a(x;h)\simeq\Oc(1).
\]
Now if $u(x;h)$ as in \eqref{CGO_0} is a solution of $(-\Delta)^m u(x;h)=0$ in $\Omega$, we see that
\begin{align*}
0 = e^{-\frac{\vp+\I\psi}{h}}(-\Delta)^m u = e^{-\frac{\vp+\I\psi}{h}} (-\Delta)^m e^{\frac{\vp+\I\psi}{h}} \left(A(x;h) + r(x;h)\right).
\end{align*}%
This implies 
\begin{align*}
e^{-\frac{\vp+\I\psi}{h}} (-\Delta)^m e^{\frac{\vp+\I\psi}{h}}r(x;h) = F(x;h), \quad \mbox{for some }F(x;h) \in L^2(\Omega), \quad \mbox{for all } h>0 \mbox{ small}.
\end{align*}
By our choices, $a_j(x)$ annihilates all the terms of order $h^{-m+j}$ in $e^{-\frac{\vp+\I\psi}{h}}\Lc(x;D) e^{\frac{\vp+\I\psi}{h}}a(x;h)$ in $\Omega$ for $j=0,\dots,m-1$. Thus we get $\|F(x,h)\|_{L^2(\Omega)} \leq C$, where $C>0$ is uniform in $h$ for $h\ll 1$.

Using Proposition \ref{Prop_Existence} we have the existence of $r(x;h) \in H^{2m}(\Omega)$ solving
\[e^{-\frac{\vp+\I\psi}{h}}(-\Delta)^m e^{\frac{\vp+\I\psi}{h}}r(x;h) = F(x;h),
\]
with the estimate
\[	\|r(x;h)\|_{H^{2m}_{\mathrm{scl}}(\Omega)} \leq C h^m, \quad \mbox{for }h>0\mbox{ small enough}.
\]
Similarly, we can construction  CGO solution of the adjoint equation  $ (e^{-\frac{\vp+\I\psi}{h}}(-\Delta)^m e^{\frac{\vp+\I\psi}{h}})^* u= e^{\frac{\vp+\I\psi}{h}}(-\Delta)^m e^{-\frac{\vp+\I\psi}{h}}u=0 $. We now sum up the above calculation in the next lemma.
\begin{lemma}\label{lm:cgo_forms}
    Let $h>0$ small enough and $\vp,\psi \in C^{\infty}(\overline{\Omega})$ satisfy $p_{0,\vp}(x,\nabla\psi) =0$. There are suitable choices of $a_0(x), \dots, a_{m-1}(x), b_0(x), \dots, b_{m-1}(x) \in C^{\infty}(\overline{\Omega})$ 
and $r(x;h), \wt{r}(x;h) \in H^{2m}(\Omega)$ such that
\begin{equation*}
\begin{aligned}
u(x;h) =& e^{\frac{\vp + \I \psi}{h}}\left(a_0(x) + ha_1(x) + \dots + h^{m-1}a_{m-1}(x) + r(x;h)\right)=  e^{\frac{\vp + \I \psi}{h}} \tilde{A},\\
v(x;h) =& e^{-\frac{\vp + \I \psi}{h}}\left(b_0(x) + hb_1(x) + \dots + h^{m-1}b_{m-1}(x) + \wt{r}(x;h)\right)=  e^{-\frac{\vp + \I \psi}{h}}\tilde{B}
\end{aligned}
\end{equation*}%
solving $(-\Delta)^m u(x)= (-\Delta)^m v(x)=0$ in $\Omega$ for $h>0$ small enough, with the estimates $\|r(x;h)\|_{H^{2m}_{\mathrm{scl}}(\Omega)}, \|\wt{r}(x;h)\|_{H^{2m}_{\mathrm{scl}}(\Omega)} \leq C h^{m}$.
Moreover, $a_0(x)$ and $b_0(x)$ solve the transport equations
\begin{align*}
    T^m a_0=T^m b_0=0.
\end{align*}%
\end{lemma}

\subsection{Solutions with linear phase}

Now we choose a suitable coordinate system and express the solutions in that  system. We choose $\vp(x)= x\cdot e_1=x_1$, where $e_1=(1,0,0\cdots,0)$. Choose an orthonormal frame $ \{\eta_1=e_1,\eta_2,\cdots,\eta_n\}$, where $\{\eta_2,\cdots,\eta_n\}$ are unit vectors on the hyperplane perpendicular to $e_1$. Then we choose $\psi(x)=x\cdot \eta_2$. We denote $y=(y_1,y_2,\cdots,y_n)$ is the coordinate  with respect to new basis and $y'=(y_2,\cdots,y_n)$, $ y''= (y_3,\cdots,y_n)$.  With this choice of $\vp$ and $\psi$ the solutions in Lemma \ref{lm:cgo_forms} take the form 
\begin{equation}\label{cgos2}
\begin{aligned}
u(x;h) =& e^{\frac{1}{h}(e_1 + \I \eta_2) \cdot x}\left(a_0(x) + ha_1(x) + \dots + h^{m-1}a_{m-1}(x) + r(x;h)\right)=  e^{\frac{1}{h}(e_1 + \I \eta_2) \cdot x} \tilde{A},\\
v(x;h) =& e^{-\frac{1}{h}(e_1 + \I \eta_2) \cdot x}\left(b_0(x) + hb_1(x) + \dots + h^{m-1}b_{m-1}(x) + \wt{r}(x;h)\right)=  e^{-\frac{1}{h}(e_1 + \I \eta_2) \cdot x}\tilde{B}
\end{aligned}
\end{equation}
The transport equation $T^ma_0=0$  becomes   \[ T^ma_0= 2^m\,(\PD_{y_1}+\I \PD_{y_2})^m a_0 =0.\]
Denote the complex variable $ z = y_1+\I y_2$. Then the above transport equation reduces to 
\begin{align*}
    \PD_{\bar{z}}^m a_0=0.
\end{align*}
A complex valued function satisfying $\PD_{\bar{z}}^m a_0=0$ is known as poly-analytic function; see \cite{polyanalytic_functions} for more details. The general solution of $\PD_{\bar{z}}^m a_0=0$ is given in the following lemma.
\begin{lemma}\label{general_solu}
   The general solution of $\PD_{\bar{z}}^m a_0=0$ is given by $ a_0=\sum\limits_{k=0}^{m-1} (z-\bar{z})^k f_{k}(z)$, where $f_k$ is a holomorphic function for all $0\le k \le m-1$.
\end{lemma}
  We do not give the proof of this lemma as it follows from standard induction argument; see  \cite[Lemma 2.6]{polyharmonic_mrt_application}. This immediately gives the following particular solution of $T^m a_0=T^m b_0=0$ having the form 
  \begin{align}\label{particular_form}
      a_0=b_0= y^k_2 f(z) g(y''), \quad \mbox{for each} \quad 0\le k\le m-1,
  \end{align}
  where $ g$ is any smooth function in $y''$ variable.
  This particular form of solution will be helpful in order to get the MRT of unknown coefficients.

\section{linearization of navier to neumann map}\label{linearization}
In this section, we linearize the Navier-to-Neumann map by computing its Fr\'echet derivative. We follow the analogous argument used in \cite[Lemma 4.4]{koch_ruland_salo_instability}.
To this end we first recall the operator \eqref{operator} together with its Navier boundary conditions.
\begin{equation}\label{operator3}
\begin{aligned}
     \begin{cases}
         	\Lc(x,D)=
		(-\Delta)^m + Q(x,D), \quad \hspace{4.3cm} \mbox{in} \quad \Omega\\
		\quad\quad\gamma u=(u,-\Delta u, \cdots, (-\Delta)^{m-1}u)=(f_0,\cdots,f_{m-1}) \qquad \mbox{on} \quad \PD\Omega,
		 \end{cases}
\end{aligned}
		\end{equation}
	where $ Q(x,D)$ is a partial differential operator of order $2m-1$ and  given by: 
	\begin{equation*}
	    Q(x,D)= \sum\limits_{l=0}^{2m-1} a^{l}_{i_1\cdots i_{l}}(x) \, D^{i_1\cdots i_l}.
	\end{equation*}
If $0$ is not an eigen value of $ \Lc(x,D)$, then \eqref{operator3} has a unique solution $ u\in H^{2m}(\Omega)$ for any $ (f_0,\cdots,f_{m-1})\in \prod_{k=0}^{m-1} H^{2m-2k-\frac{1}{2}}(\partial\Omega) $; see \cite{Polyharmonic_Book}. The Navier to Neumann map is denote by $\N_{Q}$ and defined as follows.
\begin{align}
\begin{cases}
     \N_{Q}:\prod_{k=0}^{m-1} H^{2m-2k-\frac{1}{2}}(\partial\Omega) \longrightarrow \prod_{k=0}^{m-1} H^{2m-2k-\frac{3}{2}}(\partial\Omega)   \quad \mbox{by}\\
 \N_{Q}(f_0,\cdots,f_{m-1})= ( \PD_{\nu} f_0|_{\PD\Omega}, \cdots, \PD_{\nu} f_{m-1}|_{\PD\Omega}).
\end{cases}
\end{align}
The Fr\'echet derivative of $\N_{Q}$ is given in the next lemma. 
\begin{lemma}
    Suppose $0$ is not an eigen value of $ (-\Delta)^m\,u+Q(x,D)u=0$ in $ \Omega$.  Let $P_{Q}: \prod_{k=0}^{m-1} H^{2m-2k-\frac{1}{2}}(\partial\Omega) \rightarrow H^{2m}(\Omega) $ be the solution operator for the Dirichlet problem
    \begin{equation}
        \begin{aligned}
        \begin{cases}
              ( (-\Delta)^m + Q(x,D))P_{Q}f=0  \quad \hspace{5.6cm} \mbox{in}\quad \Omega\\
     \qquad\qquad\qquad\quad\quad  \gamma P_{Q}f=(P_{Q}f, -\Delta P_Qf, \cdots, (-\Delta)^{m-1} P_{Q}f)
     \quad \mbox{on}\quad \PD\Omega.
        \end{cases}
      \end{aligned}
    \end{equation}
    
    Suppose  $G_{Q}:L^2(\Omega)\rightarrow \mathcal{D}(\Lc)$ be the Green operator satisfies 
    \begin{align*}
        ((-\Delta)^m +Q)G_{Q}F=F \quad \mbox{in}\quad \Omega, \quad \quad \gamma G_{Q}F=0 \quad \mbox{on}\quad \PD\Omega.
    \end{align*}
    Then the linearized (or Fr\'echet derivative of) Navier to Neumann map  \[B_Q=(D\N)_Q:L^{\infty}(\Omega)\longrightarrow B\left(\prod_{k=0}^{m-1} H^{2m-2k-\frac{1}{2}}(\partial\Omega) ,\prod_{k=0}^{m-1} H^{2m-2k-\frac{3}{2}}(\partial\Omega)\right)\] is given by
    \begin{align}\label{linearized_dn_map}
        (B_QH)(f)= \PD_{\nu} \gamma G_{Q}(-HP_{Q} f)|_{\PD\Omega}.
    \end{align}
\end{lemma}
\begin{proof}
Suppose $ \lVert H \rVert_{L^{\infty}(\Omega)}$ is small so that $\N_{Q+H}$ is well defined, where $H$ is same as $Q$ with different set of smooth tensor fields. Given $f=(f_0,f_1,\dots,f_{m-1})\in \prod_{k=0}^{m-1} H^{2m-2k-\frac{1}{2}}(\partial\Omega)$  we have 
\begin{align*}
    \N_{Q+H}f-\N_Q f= \PD_{\nu}(\gamma P_{Q+H}f-\gamma P_{Q}f)|_{\PD\Omega}.
\end{align*}
The function $ w\coloneqq P_{Q+H}f-P_{Q}f$ satisfies the following partial differential equation.
\begin{align*}
    \lr{(-\Delta)^m +Q(x,D)}w &= -H\,w-H\,P_{Q}f \quad \hspace{2.7cm} \mbox{in}\quad \Omega\\  \gamma w&=(w,(-\Delta) w, \cdots, (-\Delta)^{m-1} w)=0 \quad  \mbox{on} \quad \PD\Omega.
\end{align*}
\newcommand{\nrm}[1]{\lVert #1 \rVert}
We can write $w=G_{Q} (-H\,w)+ G_{Q}(-H P_{Q}f)$ and $ w\in \mathcal{D}(\Lc)$.
Utilizing the continuity of the Green operator $  G_{Q}(H\,w) $  we obtain
\begin{align*}
    \lVert G_{Q}(H\,w)\rVert_{H^{2m}(\Omega)} 
    &\le c \lVert H\,w\rVert_{L^2(\Omega)}\le \frac{1}{2} \lVert w\rVert_{H^{2m}(\Omega)}  \quad \mbox{if} \quad \lVert H\rVert_{L^{\infty}(\Omega)} \quad \mbox{is small}. 
\end{align*}
We have $ \lVert w\rv_{H^{2m}(\Omega)} = \nrm {G_{Q} (-H\,w)+ G_{Q}(-H P_{Q}f)}_{H^{2m}(\Omega)}$. The combination of this  with triangle inequality and last displayed relation implies $\lVert w\rv_{H^{2m}(\Omega)} \le \lv H\rv_{L^{\infty}(\Omega)}\,\lv f\rv_{\prod_{k=0}^{m-1} H^{2m-2k-\frac{1}{2}}(\partial\Omega)}.  $ Next we observe that,
\begin{align*}
   (\N_{Q+H}-\N_Q -D\N_{Q}(H))(f)=  \PD_{\nu} \lr{\gamma w- \gamma G_{Q}HP_{Q}f}|_{\PD\Omega}= \PD_{\nu}(\gamma G_{Q}(-Hw))|_{\PD\Omega}.
\end{align*}
By trace theorem, continuity of $G_{Q}$ and $ \lVert w\rv_{H^{2m}(\Omega)} \le \lv H\rv_{L^{\infty}(\Omega)}\,\lv f\rv_{\prod_{k=0}^{m-1} H^{2m-2k-\frac{1}{2}}(\partial\Omega)}$  we obtain from above 
\begin{align*}
   &\lVert \PD_{\nu}(\gamma G_{Q}(-Hw))\rVert_{\prod_{k=0}^{m-1} H^{2m-2k-\frac{3}{2}}(\partial\Omega)} \\ &\quad\le   \lVert G_{Q}(H\,w)\rVert_{H^{2m}(\Omega)}  \le \lv H\rv_{L^{\infty}(\Omega)}\, \lVert w\rv_{H^{2m}(\Omega)} \le  \lv H\rv^2_{L^{\infty}(\Omega)}\, \lv f\rv_{\prod_{k=0}^{m-1} H^{2m-2k-\frac{1}{2}}(\partial\Omega)}.
\end{align*}

This proves that the Fr\'echet derivative of $Q\mapsto \N_{Q}$ at $Q$ is $B_Q$.
\end{proof}
Our next result gives the required integral identity  involving unknown coefficients under the assumption that $B_{Q}=0$ at $Q=0$. Compare next result with \eqref{integral_identity} where a formal computation is given.
\begin{lemma}\label{prop:linearized_dn_map_vanish}
Assume that the linearized (or Fr\'echet derivative of) Navier to Neumann map  vanishes at $Q=0$ $i.e.,$ $(B_{0}H)f=0$. Then the following integral identity
\begin{align*}\label{integral_identity}
    \int\limits_{\Omega} \sum\limits_{l=0}^{2m-1}\, \tilde{a}^l_{i_1\cdots i_l}(x)\,D^{i_1\cdots i_l} u\, v=0 \quad \mbox{such that} \quad \Delta^m u=\Delta^m v=0 \quad \mbox{in} \quad \Omega.
\end{align*}
holds. Where $H$ is same as $Q$ represented by $\tilde{a}^l$, for certain smooth symmetric tensor fields of order $l$ for each $0\le l\le 2m-1$.
\end{lemma}
\begin{proof}
  Let $P_0g \in H^{2m}(\Omega)$ be an arbitrary function satisfying 
  \begin{align*}
      (-\Delta)^m P_{0}g=0 \quad \mbox{in}\quad \Omega \quad \mbox{and} \quad \gamma P_0g=g \quad \mbox{on} \quad \PD\Omega,
  \end{align*}
  where $ g=(g_0,\cdots,g_{m-1})\in \prod_{k=0}^{m-1} H^{2m-2k-\frac{1}{2}}(\partial\Omega).$
  Next consider the following:  \begin{align*}
      \langle (B_{0}H)f,g\rangle_{\prod_{k=0}^{m-1} H^{2m-2k-\frac{3}{2}}(\partial\Omega),\prod_{k=0}^{m-1} H^{2m-2k-\frac{1}{2}}(\partial\Omega)}= \int\limits_{\PD\Omega} \PD_{\nu}\gamma G_{0}(-HP_0f)\, g
  \end{align*}
  where $ G_{0}(-HP_0f)$ solves $(-\Delta)^mG_0(-HP_0h)= -H\,P_0f $ in $ \Omega$ and  $\gamma G_{0}(-HP_0f)|_{\PD\Omega}=0$. Multiplying $P_0g$ to the equation $(-\Delta)^mG_0(-HP_0h)= -H\,P_0f$ we get
  \begin{align*}
   \int\limits_{\Omega}  (-\Delta)^mG_0(-HP_0h)\, P_0g&= - \int\limits_{\Omega} P_0g\, H(P_0f).
   \end{align*}
  We next make use of integration parts,  $ (-\Delta)^m P_0g=0$ in $ \Omega $ and $\gamma G_{0}(-HP_0f)=0$ on $\PD\Omega$ to derive
   \begin{align*}
   \int\limits_{\Omega}  (-H\,P_0f)\, P_0g&=    \sum\limits_{k=0}^{m-1} \int\limits_{\PD\Omega}\PD_{\nu}(-\Delta)^{m-k-1} G_{0}(-HP_0f)\, (-\Delta)^{k}P_0g\nonumber\\
   &= \langle B_{0}(f),g\rangle_{\prod_{k=0}^{m-1} H^{2m-2k-\frac{3}{2}}(\partial\Omega),\prod_{k=0}^{m-1} H^{2m-2k-\frac{1}{2}}(\partial\Omega)}.
   \end{align*}
   This along with  $(B_0H)f=0$ implies
      $ \int\limits_{\Omega}  (-H\,P_0f)\, P_0g=0$,
  where $ P_0f$ and $P_0g$ solve $(-\Delta)^m\, (\cdot)=0$. Recall that $H$ is same as $Q $ represented by $\tilde{a}^l$.   This finishes the proof.
\end{proof}

\section*{Acknowledgements}

Both authors were partly supported by the Academy of Finland (Centre of Excellence in Inverse Modelling and Imaging, grant 284715) and by the European Research Council under Horizon 2020 (ERC CoG 770924).

\bibliographystyle{alpha}
\bibliography{bibfile.bib}

\begin{bibdiv}
\begin{biblist}

\bib{polyanalytic_functions}{book}{
      author={Balk, Mark~Benevich},
       title={Polyanalytic functions},
      series={Mathematical Research},
   publisher={Akademie-Verlag, Berlin},
        date={1991},
      volume={63},
        ISBN={3-05-501292-5},
      review={\MR{1184141}},
}

\bib{BG_19}{article}{
      author={Bhattacharyya, Sombuddha},
      author={Ghosh, Tuhin},
       title={Inverse boundary value problem of determining up to a second
  order tensor appear in the lower order perturbation of a polyharmonic
  operator},
        date={2019},
        ISSN={1069-5869},
     journal={J. Fourier Anal. Appl.},
      volume={25},
      number={3},
       pages={661\ndash 683},
         url={https://doi.org/10.1007/s00041-018-9625-3},
      review={\MR{3953481}},
}

\bib{BG20}{misc}{
      author={Bhattacharyya, Sombuddha},
      author={Ghosh, Tuhin},
       title={An inverse problem on determining second order symmetric tensor
  for perturbed biharmonic operator},
        date={2020},
        note={https://arxiv.org/abs/2010.00192},
}

\bib{Brown_Gauthier}{misc}{
      author={Brown, R.~M.},
      author={Gauthier, L.~D.},
       title={Inverse boundary value problems for polyharmonic operators with
  non-smooth coefficients},
   publisher={arXiv},
        date={2021},
         url={https://arxiv.org/abs/2108.11522},
        note={https://arxiv.org/abs/2108.11522},
}

\bib{polyharmonic_mrt_application}{misc}{
      author={Bhattacharyya, Sombuddha},
      author={Krishnan, Venkateswaran~P.},
      author={Sahoo, Suman~Kumar},
       title={Unique determination of anisotropic perturbations of a
  polyharmonic operator from partial boundary data},
        date={2021},
        note={https://arxiv.org/abs/2111.07610},
}

\bib{BUK}{article}{
      author={Bukhgeim, Alexander~L.},
      author={Uhlmann, Gunther},
       title={Recovering a potential from partial {C}auchy data},
        date={2002},
        ISSN={0360-5302},
     journal={Comm. Partial Differential Equations},
      volume={27},
      number={3-4},
       pages={653\ndash 668},
         url={https://doi.org/10.1081/PDE-120002868},
      review={\MR{1900557}},
}

\bib{Calderon1980}{incollection}{
      author={Calder\'{o}n, Alberto-P.},
       title={On an inverse boundary value problem},
        date={1980},
   booktitle={Seminar on {N}umerical {A}nalysis and its {A}pplications to
  {C}ontinuum {P}hysics ({R}io de {J}aneiro, 1980)},
   publisher={Soc. Brasil. Mat., Rio de Janeiro},
       pages={65\ndash 73},
      review={\MR{590275}},
}

\bib{Catalin_Ali}{misc}{
      author={Cârstea, Cătălin~I.},
      author={Feizmohammadi, Ali},
       title={A density property for tensor products of gradients of harmonic
  functions and applications},
   publisher={arXiv},
        date={2020},
         url={https://arxiv.org/abs/2009.11217},
        note={https://doi.org/10.48550/arxiv.2009.11217},
}

\bib{Dairbekov_Sharafutdinov}{article}{
      author={Dairbekov, Nurlan~S.},
      author={Sharafutdinov, Vladimir~A.},
       title={Conformal {K}illing symmetric tensor fields on {R}iemannian
  manifolds},
        date={2010},
        ISSN={1560-750X},
     journal={Mat. Tr.},
      volume={13},
      number={1},
       pages={85\ndash 145},
         url={https://doi.org/10.3103/s1055134411010019},
      review={\MR{2682769}},
}

\bib{DOS}{article}{
      author={Dos Santos~Ferreira, D.},
      author={Kenig, C.},
      author={Sj\"{o}strand, J.},
      author={Uhlmann, G.},
       title={Determining a magnetic schr\"{o}dinger operator from partial
  cauchy data},
        date={2007},
     journal={Comm. Math. Phys.},
      volume={2},
       pages={467\ndash 488},
}

\bib{linearized_partial_data}{article}{
      author={Dos Santos~Ferreira, David},
      author={Kenig, Carlos~E.},
      author={Sj\"{o}strand, Johannes},
      author={Uhlmann, Gunther},
       title={On the linearized local {C}alder\'{o}n problem},
        date={2009},
        ISSN={1073-2780},
     journal={Math. Res. Lett.},
      volume={16},
      number={6},
       pages={955\ndash 970},
         url={https://doi.org/10.4310/MRL.2009.v16.n6.a4},
      review={\MR{2576684}},
}

\bib{Evans_pde_book}{book}{
      author={Evans, Lawrence~C.},
       title={Partial differential equations},
     edition={Second},
      series={Graduate Studies in Mathematics},
   publisher={American Mathematical Society, Providence, RI},
        date={2010},
      volume={19},
        ISBN={978-0-8218-4974-3},
         url={https://doi.org/10.1090/gsm/019},
      review={\MR{2597943}},
}

\bib{Polyharmonic_Book}{book}{
      author={Gazzola, Filippo},
      author={Grunau, Hans-Christoph},
      author={Sweers, Guido},
       title={Polyharmonic boundary value problems},
      series={Lecture Notes in Mathematics},
   publisher={Springer-Verlag, Berlin},
        date={2010},
      volume={1991},
        ISBN={978-3-642-12244-6},
         url={http://dx.doi.org/10.1007/978-3-642-12245-3},
        note={Positivity preserving and nonlinear higher order elliptic
  equations in bounded domains},
      review={\MR{2667016 (2011h:35001)}},
}

\bib{Ghosh15}{article}{
      author={Ghosh, Tuhin},
       title={An inverse problem on determining upto first order perturbations
  of a fourth order operator with partial boundary data},
        date={2015},
        ISSN={0266-5611},
     journal={Inverse Problems},
      volume={31},
      number={10},
       pages={105009, 19},
         url={https://doi.org/10.1088/0266-5611/31/10/105009},
      review={\MR{3405369}},
}

\bib{Ghosh-Krishnan}{article}{
      author={Ghosh, Tuhin},
      author={Krishnan, Venkateswaran~P.},
       title={Determination of lower order perturbations of the polyharmonic
  operator from partial boundary data},
        date={2016},
        ISSN={0003-6811},
     journal={Appl. Anal.},
      volume={95},
      number={11},
       pages={2444\ndash 2463},
         url={https://doi.org/10.1080/00036811.2015.1092522},
      review={\MR{3546596}},
}

\bib{linearzed_complex_manifold}{article}{
      author={Guillarmou, Colin},
      author={Salo, Mikko},
      author={Tzou, Leo},
       title={The linearized {C}alder\'{o}n problem on complex manifolds},
        date={2019},
        ISSN={1439-8516},
     journal={Acta Math. Sin. (Engl. Ser.)},
      volume={35},
      number={6},
       pages={1043\ndash 1056},
         url={https://doi.org/10.1007/s10114-019-8129-7},
      review={\MR{3952702}},
}

\bib{linearized_KLS}{misc}{
      author={Krupchyk, Katya},
      author={Liimatainen, Tony},
      author={Salo, Mikko},
       title={Linearized calder\'on problem and exponentially accurate
  quasimodes for analytic manifolds},
        date={2020},
        note={https://arxiv.org/abs/2009.05699},
}

\bib{KRU2}{article}{
      author={Krupchyk, Katsiaryna},
      author={Lassas, Matti},
      author={Uhlmann, Gunther},
       title={Determining a first order perturbation of the biharmonic operator
  by partial boundary measurements},
        date={2012},
        ISSN={0022-1236},
     journal={J. Funct. Anal.},
      volume={262},
      number={4},
       pages={1781\ndash 1801},
         url={https://doi.org/10.1016/j.jfa.2011.11.021},
      review={\MR{2873860}},
}

\bib{KRU1}{article}{
      author={Krupchyk, Katsiaryna},
      author={Lassas, Matti},
      author={Uhlmann, Gunther},
       title={Inverse boundary value problems for the perturbed polyharmonic
  operator},
        date={2014},
        ISSN={0002-9947},
     journal={Trans. Amer. Math. Soc.},
      volume={366},
      number={1},
       pages={95\ndash 112},
         url={https://doi.org/10.1090/S0002-9947-2013-05713-3},
      review={\MR{3118392}},
}

\bib{koch_ruland_salo_instability}{misc}{
      author={Koch, Herbert},
      author={Rüland, Angkana},
      author={Salo, Mikko},
       title={On instability mechanisms for inverse problems},
        date={2021},
        note={https://arxiv.org/abs/2012.01855},
}

\bib{Kenig_Salo_Survey}{incollection}{
      author={Kenig, Carlos},
      author={Salo, Mikko},
       title={Recent progress in the {C}alder\'on problem with partial data},
        date={2014},
   booktitle={Inverse problems and applications},
      series={Contemp. Math.},
      volume={615},
   publisher={Amer. Math. Soc., Providence, RI},
       pages={193\ndash 222},
         url={http://dx.doi.org/10.1090/conm/615/12245},
      review={\MR{3221605}},
}

\bib{KEN}{article}{
      author={Kenig, Carlos~E.},
      author={Sj\"{o}strand, Johannes},
      author={Uhlmann, Gunther},
       title={The {C}alder\'{o}n problem with partial data},
        date={2007},
        ISSN={0003-486X},
     journal={Ann. of Math. (2)},
      volume={165},
      number={2},
       pages={567\ndash 591},
         url={https://doi.org/10.4007/annals.2007.165.567},
      review={\MR{2299741}},
}

\bib{kuchment_cone_trans}{article}{
      author={Kuchment, Peter},
      author={Terzioglu, Fatma},
       title={Inversion of weighted divergent beam and cone transforms},
        date={2017},
        ISSN={1930-8337},
     journal={Inverse Probl. Imaging},
      volume={11},
      number={6},
       pages={1071\ndash 1090},
         url={https://doi.org/10.3934/ipi.2017049},
      review={\MR{3708176}},
}

\bib{KrupchykUhlmann2020}{article}{
      author={Krupchyk, Katya},
      author={Uhlmann, Gunther},
       title={A remark on partial data inverse problems for semilinear elliptic
  equations},
        date={2020},
        ISSN={0002-9939},
     journal={Proc. Amer. Math. Soc.},
      volume={148},
      number={2},
       pages={681\ndash 685},
         url={https://doi.org/10.1090/proc/14844},
      review={\MR{4052205}},
}

\bib{LLLS2021}{article}{
      author={Lassas, Matti},
      author={Liimatainen, Tony},
      author={Lin, Yi-Hsuan},
      author={Salo, Mikko},
       title={Partial data inverse problems and simultaneous recovery of
  boundary and coefficients for semilinear elliptic equations},
        date={2021},
        ISSN={0213-2230},
     journal={Rev. Mat. Iberoam.},
      volume={37},
      number={4},
       pages={1553\ndash 1580},
         url={https://doi.org/10.4171/rmi/1242},
      review={\MR{4269409}},
}

\bib{Rohit_Suman}{article}{
      author={Mishra, Rohit~Kumar},
      author={Sahoo, Suman~Kumar},
       title={Injectivity and range description of integral moment transforms
  over {$m$}-tensor fields in {$\Bbb{R}^n$}},
        date={2021},
        ISSN={0036-1410},
     journal={SIAM J. Math. Anal.},
      volume={53},
      number={1},
       pages={253\ndash 278},
         url={https://doi.org/10.1137/20M1347589},
      review={\MR{4198570}},
}

\bib{real_principal_inv_prob}{misc}{
      author={Oksanen, Lauri},
      author={Salo, Mikko},
      author={Stefanov, Plamen},
      author={Uhlmann, Gunther},
       title={Inverse problems for real principal type operators},
        date={2020},
        note={https://arxiv.org/abs/2001.07599},
}

\bib{PSU_book}{book}{
      author={Paternain, Gabriel},
      author={Salo, Mikko},
      author={Uhlmann, Gunther},
       title={Geometric inverse problems with emphasis on two dimensions},
}

\bib{Sharafutdinov_1986_momentum}{article}{
      author={Sharafutdinov, Vladimir~A.},
       title={A problem of integral geometry for generalized tensor fields on
  {${\bf R}^n$}},
        date={1986},
        ISSN={0002-3264},
     journal={Dokl. Akad. Nauk SSSR},
      volume={286},
      number={2},
       pages={305\ndash 307},
      review={\MR{823390}},
}

\bib{Sharafutdinov_book}{book}{
      author={Sharafutdinov, Vladimir~A.},
       title={Integral geometry of tensor fields},
      series={Inverse and Ill-posed Problems Series},
   publisher={VSP, Utrecht},
        date={1994},
        ISBN={90-6764-165-0},
         url={https://doi.org/10.1515/9783110900095},
      review={\MR{1374572}},
}

\bib{Salo-Tzou}{article}{
      author={Salo, Mikko},
      author={Tzou, Leo},
       title={Carleman estimates and inverse problems for {D}irac operators},
        date={2009},
        ISSN={0025-5831},
     journal={Math. Ann.},
      volume={344},
      number={1},
       pages={161\ndash 184},
         url={https://doi.org/10.1007/s00208-008-0301-9},
      review={\MR{2481057}},
}

\bib{SjostrandUhlmann2016}{article}{
      author={Sj\"{o}strand, Johannes},
      author={Uhlmann, Gunther},
       title={Local analytic regularity in the linearized {C}alder\'{o}n
  problem},
        date={2016},
        ISSN={2157-5045},
     journal={Anal. PDE},
      volume={9},
      number={3},
       pages={515\ndash 544},
         url={https://doi.org/10.2140/apde.2016.9.515},
      review={\MR{3518528}},
}

\bib{Taylor_book}{book}{
      author={Taylor, Michael~E.},
       title={Partial differential equations {I}. {B}asic theory},
     edition={Second},
      series={Applied Mathematical Sciences},
   publisher={Springer, New York},
        date={2011},
      volume={115},
        ISBN={978-1-4419-7054-1},
         url={https://doi.org/10.1007/978-1-4419-7055-8},
      review={\MR{2744150}},
}

\bib{Uhl_eip_survey}{article}{
      author={Uhlmann, Gunther},
       title={Electrical impedance tomography and {C}alder\'{o}n's problem},
        date={2009},
        ISSN={0266-5611},
     journal={Inverse Problems},
      volume={25},
      number={12},
       pages={123011, 39},
         url={https://doi.org/10.1088/0266-5611/25/12/123011},
      review={\MR{3460047}},
}

\bib{Uhlmann_survey}{incollection}{
      author={Uhlmann, Gunther},
       title={30 years of {C}alder\'{o}n's problem},
        date={2014},
   booktitle={S\'{e}minaire {L}aurent {S}chwartz---\'{E}quations aux
  d\'{e}riv\'{e}es partielles et applications. {A}nn\'{e}e 2012--2013},
      series={S\'{e}min. \'{E}qu. D\'{e}riv. Partielles},
   publisher={\'{E}cole Polytech., Palaiseau},
       pages={Exp. No. XIII, 25},
      review={\MR{3381003}},
}

\end{biblist}
\end{bibdiv}

\end{document}